\documentclass[11pt]{article}
\normalsize
\usepackage{amsmath,amssymb,amsthm,mathtools}
\usepackage{enumitem}
\usepackage{stmaryrd}

\usepackage{hyperref}

{\theoremstyle{plain}
\newtheorem{theorem}{Theorem}[section]
\newtheorem{lemma}[theorem]{Lemma}
\newtheorem{corollary}[theorem]{Corollary}
\newtheorem{proposition}[theorem]{Proposition}}

{\theoremstyle{definition}
\newtheorem{remark}[theorem]{Remark}
\newtheorem{definition}[theorem]{Definition}
\newtheorem{example}[theorem]{Example}
\newtheorem{question}[theorem]{Question}}

\begin{document}
\title{Projector additive group codes}
\author{Javier de la Cruz\\Universidad del Norte, Barranquilla, Colombia}
\date{}
\maketitle

\begin{abstract}
Let $F=\mathbb{F}_q$ and let $K=\mathbb{F}_{q^m}$ be a finite extension of
$F$. An additive left group code is a left $FG$-submodule of the group
algebra $KG$. Classical idempotent group codes describe images of
$KG$-linear projectors and are necessarily $K$-linear. They therefore do not
provide a sufficiently broad framework for additive group codes, which need
not be $K$-linear. In this paper, we develop a projector-based approach for
the study of additive left group codes.
We distinguish between restricted idempotent additive group codes of the form
$FGe$, where $e\in KG$ is an idempotent, and projector additive left group
codes, which are images of arbitrary $FG$-linear projectors on $KG$. While
$KG$-linear projector group codes coincide with classical idempotent group
codes, their additive counterparts do not, in general, coincide. We show
that projector additive group codes provide constructions beyond the
restricted idempotent setting, and we give examples in both the semisimple
and non-semisimple cases with parameters not attained by restricted
idempotent additive group codes.

The projector framework also provides effective algebraic tools for studying
duality. We relate trace-Euclidean and trace-Hermitian duality to adjoints of
$FG$-linear projectors, characterize LCD additive group codes through
self-adjoint projectors, and obtain sufficient conditions for self-duality.
We further study Murray--von Neumann equivalence of projectors and show that
it characterizes isomorphism of their images as left $FG$-modules, thereby
providing a module-theoretic classification tool for projector additive group
codes. Finally, we interpret quotients by orthogonal codes in terms of module
duals.
\end{abstract}

\section{Introduction}

Let $F=\mathbb{F}_q$ and let $K=\mathbb{F}_{q^m}$ be a finite extension of
degree $m$. An additive code of length $n$ over $K$ is an $F$-subspace of
$K^n$. In particular, every $K$-linear code is additive, whereas the converse
does not hold in general. Thus, additive codes provide a broader class than
linear codes over $K$. Additive codes over $\mathbb{F}_4$ were introduced in
\cite{Calderbank} because of their connection with quantum error-correcting
codes. More precisely, the construction of quantum stabilizer codes can be
reduced to the study of additive codes over $\mathbb{F}_4$ that are
self-orthogonal with respect to a suitable trace inner product. Additive codes
over $\mathbb{F}_9$ that are self-dual with respect to the trace-Hermitian
inner product were later classified in \cite{Danielsen}, where they were
related to ternary quantum error-correcting codes. More general classes of
additive codes have since been studied in
\cite{Bierbrauer-7, Huffman-2010, Huffman-2013}.

Let $G$ be a finite group of order $n$. Via the standard identification of
$KG$ with $K^n$, the group algebra $KG$ provides a natural algebraic setting
for codes of length $n$. In this context, an additive group code is a left
$FG$-submodule of $KG$. Likewise, a linear group code over $K$ is a left
$KG$-submodule of $KG$, equivalently, a left ideal of $KG$. Hence every
linear group code is an additive group code, but the converse need not hold.
Since $KG$ may be noncommutative, left and right versions must in general be
distinguished. Throughout this paper, unless explicitly stated otherwise, all
additive group codes are understood to be left additive group codes.
Additive group codes over extensions of finite chain rings, together with
their decomposition theory and trace-Euclidean duality, have recently been
investigated in \cite{BCFM-2025}.

Classical idempotent group codes are the left ideals of the form $C=KGe$,
where $e^2=e\in KG$ is an idempotent. Equivalently, they are precisely the images
of $KG$-linear projectors on $KG$. Thus, in the classical $K$-linear setting,
the idempotent and projector constructions coincide. 
The situation is different for additive group codes, as will become clear
below. First, the direct additive
analogue of the classical idempotent construction is the restricted
idempotent additive group code $FGe$, where $e\in KG$ is an idempotent. For this reason, and whenever no confusion is likely to arise, we
shall occasionally refer to a code of the form $FGe$ as an idempotent
additive group code. On the other hand, the additive
analogue of the classical construction of codes as images of $KG$-linear
projectors is provided by projector additive left group codes. More
precisely, a projector additive left group code is a code of the form
$\operatorname{Im}(P)$, where $P\in\operatorname{End}_{FG}(KG)$ is an
$FG$-linear projector.

Thus, the restricted idempotent construction is obtained by restricting to
$FG$ a projector induced by an idempotent of $KG$, whereas the projector
construction allows arbitrary $FG$-linear projectors on $KG$. In this way,
projector additive left group codes extend to the additive setting the
classical construction of codes as images of projectors.

Unlike the classical $K$-linear setting, the restricted idempotent and
projector constructions do not, in general, coincide in the additive
framework. If $FG$ is semisimple, then every additive left group code is a
projector additive left group code, whereas restricted idempotent additive
group codes form a subclass that can be proper. More generally, in both the
semisimple and non-semisimple settings, projector additive group codes need
not be of the form $FGe$ for any idempotent $e\in KG$. Examples of both
phenomena are given in Examples~\ref{ex:semisimple-projector} and
\ref{ex:nonsemisimple-projector}. Thus, the projector approach provides a
genuinely different and more flexible construction than the restricted
idempotent construction.

This distinction is also relevant from the coding-theoretic point of view.
In both the semisimple and non-semisimple settings, projector additive group
codes can yield genuinely additive codes with parameters not attained by the
restricted idempotent construction. In particular,
Example~\ref{ex:semisimple-projector} gives a semisimple example of a
projector additive group code with strictly larger minimum distance than every
restricted idempotent additive group code of the same length and cardinality.
Likewise, Example~\ref{ex:nonsemisimple-projector} shows that the same
phenomenon can occur in the non-semisimple case. Thus, the projector
construction is not only structurally more flexible, but may also lead to
improved coding-theoretic parameters.

A further motivation for studying projector additive group codes comes from
their duality properties. We develop a projector-adjoint approach to
trace-Euclidean duality and, when $m$ is even, to trace-Hermitian duality.
This approach relates the dual of a projector additive group code to the
adjoint of a defining $FG$-linear projector. In particular, we characterize
trace-LCD additive left group codes by the existence of a self-adjoint
$FG$-linear projector having the code as its image; see
Theorem~\ref{thm:LCD-projector-characterization}. As an application of this
result, we recover the classical characterization of Euclidean LCD left group
codes as left ideals of the form $KGe$, where $e\in KG$ is a self-adjoint
idempotent. When $m$ is even, we similarly recover the corresponding
characterization for Hermitian LCD left group codes. These classical results
are known from \cite{CruzWillems2017,BBW24}.

We also show that, if a defining projector $P$ satisfies $P^*=I-P$ with
respect to the corresponding trace form, then its image is trace self-dual;
see Proposition~\ref{prop:projector-selfdual}. Thus, LCD and self-dual
additive group codes can be investigated by means of algebraic properties of
their defining projectors. To the best of our knowledge, this
projector-adjoint approach to trace duality for additive group codes has not
previously been developed in the literature.

We also investigate the Murray--von Neumann equivalence of $FG$-linear
projectors. This equivalence has a natural module-theoretic interpretation:
if $C=\operatorname{Im}(P)$ and $D=\operatorname{Im}(Q)$ are projector
additive group codes, then $C$ and $D$ are isomorphic as left $FG$-modules if
and only if $P$ and $Q$ are Murray--von Neumann equivalent. Accordingly,
this notion provides a tool for organizing and classifying projector
additive group codes according to their $FG$-module structure. We emphasize
that this is not, in general, a metric equivalence: it need not preserve
Hamming weight, minimum distance, or weight distribution. Finally, we
interpret quotients by orthogonal codes in terms of module duals, thereby
connecting trace duality with the representation-theoretic structure of
additive group codes.

The main objective of this paper is therefore to establish a projector-based
theory for additive left group codes. We study the module-theoretic structure
of projector and restricted projector additive group codes, their
trace-Euclidean and trace-Hermitian duality properties, their LCD and
self-dual instances, and their classification up to Murray--von Neumann
equivalence. Several of our results may be viewed as additive counterparts or
extensions of results obtained in
\cite{BCFM-2025, BBW24, CruzWillems2017, CruzWillems2025,
DeLaCruzWillems2021, Willems2002} for LCD and self-dual group codes, as well
as for twisted and skew-twisted group codes.

The paper is organized as follows. Section~2 contains the necessary
preliminaries on $\sigma$-sesquilinear forms, additive codes in $K^n$, and
additive group codes in $KG$. Section~3 develops the projector approach to
additive group codes, including $FG$-projectors and their adjoints,
restricted projector codes, LCD and self-dual projector additive group codes,
and Murray--von Neumann equivalence. Section~4 is devoted to module duals and
quotients by orthogonal codes.

\section{Preliminaries}
Throughout the paper, the following notation is used:
\begin{itemize}
    \item $G$ is a finite group, $F=\mathbb{F}_q$, and $K=\mathbb{F}_{q^m}$ is a finite extension of degree $m$.
    \item $KG$ and $FG$ are the group algebras of $G$ over $K$ and $F$, respectively.
    \item If $a=\sum_{g\in G} a_g g\in KG$, then $a^\ast=\sum_{g\in G} a_g g^{-1}$ denotes the involution of $a$ induced by inversion in $G$.
    \item If $a=\sum_{g\in G} a_g g\in KG$, then $\operatorname{coef}_{1_G}(a)=a_{1_G}$ denotes the coefficient of $1_G$ in $a$.
    \item $C^{\perp_\star}$ is the orthogonal of $C \leq KG$ with respect to $\langle\cdot,\cdot\rangle_\star$.
    \item $P^{\ast_\star}$ is the adjoint of $P\in\operatorname{End}_{FG}(KG)$ with respect to $\langle\cdot,\cdot\rangle_\star$.
    \item $e \sim_{\mathrm{MvN}} f$ denotes Murray--von Neumann equivalence of idempotents $e,f \in A$.
    \item $C^*=\operatorname{Hom}_F(C,F)$ is the module dual of an additive left group code $C$. In particular, $C^* \neq \{c^*:c\in C\}$.
    \end{itemize}
\subsection{Basic facts on $\sigma$-sesquilinear forms}

Let $F=\mathbb{F}_q$ and let $K=\mathbb{F}_{q^m}$ be a finite extension of degree $m$, let $\sigma\in\operatorname{Aut}(F)$, and let $V$ be a finite-dimensional $K$-vector space. A \textit{$\sigma$-sesquilinear form on $V$ over $F$} is a map $\langle\cdot,\cdot\rangle:V\times V\longrightarrow F$ such that, for all $x,y,z\in V$ and all $\alpha\in F$, one has
$\langle x+y,z\rangle=\langle x,z\rangle+\langle y,z\rangle$,
$\langle x,y+z\rangle=\langle x,y\rangle+\langle x,z\rangle$,
$\langle \alpha x,y\rangle=\alpha\langle x,y\rangle$,
and
$\langle x,\alpha y\rangle=\sigma(\alpha)\langle x,y\rangle$.

When $\sigma=\operatorname{id}_F$, the form is called \emph{bilinear}. If, in addition, it is \emph{symmetric}, that is, if $\langle x,y\rangle=\langle y,x\rangle$ for all $x,y\in V$, then the form is called \emph{Euclidean}.

If $\operatorname{ord}(\sigma)=2$ and $\sigma(\langle x,y\rangle)=\langle y,x\rangle$ for all $x,y\in V$, then the form is called \textit{Hermitian}.

A $\sigma$-sesquilinear form is called \textit{reflexive} if $\langle x,y\rangle=0$ if and only if $\langle y,x\rangle=0$ for all $x,y\in V$. In particular, Hermitian and Euclidean forms are reflexive.

Finally, a $\sigma$-sesquilinear form is called \textit{non-degenerate} if $\langle x,y\rangle=0$ for all $y\in V$ implies $x=0$.

For any $F$-subspace $W\leq V$, the \textit{dual space} of $W$ with respect to a reflexive $\sigma$-sesquilinear form $\star=\langle\cdot,\cdot\rangle$, denoted by $W^{\perp_{\star}}$, is defined by
$$
W^{\perp_{\star}}:=\{v\in V:\langle v,w\rangle=0 \text{ for all } w\in W\}.
$$
 It is well known that if $\star=\langle\cdot,\cdot\rangle$ is non-degenerate and reflexive, then
$
\dim_F(W^{\perp_{\star}})=\dim_F(V)-\dim_F(W)
$
and
$
(W^{\perp_{\star}})^{\perp_{\star}}=W.
$

The radical of a reflexive $\sigma$-sesquilinear form $\langle\cdot,\cdot\rangle$ on $V$ is the $F$-vector space
$$
\operatorname{Rad}(\langle\cdot,\cdot\rangle)=\{x\in V:\langle x,y\rangle=0 \text{ for all } y\in V\}.
$$
Clearly, a reflexive $\sigma$-sesquilinear form is non-degenerate if and only if its radical is trivial.

If $\mathcal{B}=\{\beta_1,\ldots,\beta_r\}$ is a basis of $V$ over $F$ and if $\langle\cdot,\cdot\rangle$ is a $\sigma$-sesquilinear form on $V$ over $F$, then for
$$
x=\sum_{i=1}^r x_i\beta_i,\qquad y=\sum_{j=1}^r y_j\beta_j,
$$
we have
$$
\langle x,y\rangle
=
\sum_{i=1}^r x_i\sum_{j=1}^r \langle \beta_i,\beta_j\rangle \sigma(y_j).
$$
If $[x]_{\mathcal B}=(x_1,\ldots,x_r)$ and $[y]_{\mathcal B}=(y_1,\ldots,y_r)$ denote the coordinate row vectors of $x$ and $y$ with respect to $\mathcal B$, then
$$
\langle x,y\rangle=[x]_{\mathcal B}\,G_{\mathcal B}\,\sigma([y]_{\mathcal B})^t,
$$
where
$$
G_{\mathcal B}=(\langle \beta_i,\beta_j\rangle)\in M_r(F).
$$
The matrix $G_{\mathcal B}$ is called the Gram matrix of $\langle\cdot,\cdot\rangle$ with respect to the basis $\mathcal B$.

\subsection{Additive codes in $K^n$}

Let $F=\mathbb{F}_q$ and let $K=\mathbb{F}_{q^m}$ be a finite extension of degree $m$. Then $K^n$ is naturally an $F$-vector space of dimension $mn$. A \textit{linear code} over $K$ is a $K$-subspace of $K^n$. An \textit{$F$-additive code} over $K$ is an $F$-subspace of $K^n$. Therefore, every linear code over $K$ is an $F$-additive code, whereas an $F$-additive code need not be $K$-linear. Observe that any additive subgroup of $K^n$ is automatically linear over the prime subfield of $K$. In this work, however, we restrict our attention to codes that are linear over the fixed subfield $F=\mathbb{F}_q$. For brevity, when no confusion is likely to arise, an $F$-additive code will simply be called an \textit{additive code}.

In particular, every $F$-additive code is an additive subgroup of $(K^n,+)$. Unless otherwise stated, the \textit{dimension} of an $F$-additive code $C\leq K^n$ will always mean $\dim_F(C)$. Hence, if $\dim_F(C)=r$, then
$
|C|=q^r.
$
Equivalently, if $|C|=(q^m)^k$, then
$
k=\frac{\dim_F(C)}{m}
$
is called the dimension of $C$ over $K$. In general, $k$ need not be an integer.

For a vector $x=(x_1,\dots,x_n)\in K^n$, the \textit{Hamming weight} of $x$ is defined by
$$
\operatorname{wt}(x)=|\{i\in\{1,\dots,n\}:x_i\neq 0\}|,
$$
and the \textit{Hamming distance} between $x,y\in K^n$ is defined by
$
d(x,y)=\operatorname{wt}(x-y).
$
The \textit{minimum distance} of a nonzero $F$-additive code $C\leq K^n$ is defined by
$$
d(C)=\min\{\operatorname{wt}(x):x\in C,\ x\neq 0\}.
$$

If $C$ has cardinality $q^r$ and minimum distance $d$, we say that $C$ has parameters $(n,q^r,d)$.

 The Singleton bound for an $F$-additive code $C\leq K^n$ with parameters
$(n,q^r,d)$ is $|C|\leq |K|^{n-d+1}$. Since $|C|=q^r$ and $|K|=q^m$, this
bound can be written as $r\leq m(n-d+1)$. Equivalently, one has
$d\leq n-\left\lceil r/m\right\rceil+1$.

We say that $C$ is \emph{Singleton-optimal} if it attains the integral
Singleton bound, namely, if $d=n-\left\lceil r/m\right\rceil+1$. If
$m\mid r$, then $C$ is MDS if and only if it is Singleton-optimal. Indeed,
writing $r=mk$, the Singleton-optimality condition becomes the usual MDS
condition $d=n-k+1$.

If $m\nmid r$, then equality in $r\leq m(n-d+1)$ is impossible, and hence
$C$ cannot be MDS in the strict sense. In this case, we call $C$
\emph{quasi-MDS}, or \emph{QMDS}, if it is Singleton-optimal. Equivalently,
when $m\nmid r$, the code $C$ is QMDS if
$d=n-\left\lceil r/m\right\rceil+1$.

\subsection{Additive group codes}
Let $G$ be a finite group. We write $KG$ for the group algebra of $G$ over $K$, that is, the $K$-vector space with basis $G$ whose multiplication is given by
$$
\left(\sum_{g\in G} a_g g\right)\left(\sum_{h\in G} b_h h\right)=\sum_{g,h\in G} a_g b_h\,gh,
$$
for all $a_g,b_h\in K$. Similarly, $FG$ denotes the group algebra of $G$ over $F$.

Let $\ast:G\longrightarrow G$ be the involution defined by $g^\ast=g^{-1}$ for all $g\in G$. Extending $\ast$ linearly, we obtain a $K$-linear involutive antiautomorphism of $KG$, still denoted by $\ast$, given by
$
\left(\sum_{g\in G} a_g g\right)^\ast=\sum_{g\in G} a_g g^{-1}.
$
In particular, for all $x,y\in KG$, one has
$
(x^\ast)^\ast=x
$
and
$
(xy)^\ast=y^\ast x^\ast.
$

Let $n=|G|$, and let $E=\{e_1,\dots,e_n\}$ be the canonical basis of $K^n$.

\begin{definition}
Let $C\leq K^n$ be an $F$-additive code. We say that $C$ is an \emph{additive left $G$-code} (respectively, \emph{additive right $G$-code}; \emph{additive $G$-code}) if there exists a bijection
$$
\varphi:E\longrightarrow G
$$
such that the $K$-linear extension $\varphi:K^n\longrightarrow KG$ maps $C$ onto a left (respectively, right; two-sided) $FG$-submodule of $KG$.
\end{definition}

This leads to the following intrinsic notion in the group algebra setting.

\begin{definition}
A subset $C\subseteq KG$ is called an \emph{additive left group code}
(respectively, an \emph{additive right group code}; an \emph{additive
group code}) if $C$ is a left $FG$-submodule (respectively, a right
$FG$-submodule; an $FG$-subbimodule) of $KG$.
\end{definition}

If $C$ is an additive left group code, we write $C\leq KG$ to indicate that
$C$ is a left $FG$-submodule of $KG$. Unless explicitly stated otherwise,
all additive group codes considered in this paper are additive left group
codes.

Note that an additive left group code $C \leq KG$ is, in particular, an $F$-vector subspace of $KG$, although not necessarily a $K$-subspace. Thus, additive left group codes form a distinguished class of additive codes, with the additional advantage that they can be studied using algebraic tools such as module theory and the representation theory of finite groups.

To define the Hamming weight, minimum distance, code parameters, and the relevant bilinear forms on $KG$, fix an ordering
$
G=\{g_1,\dots,g_n\}
$
and let
$
\psi:K^n\to KG
$
be the $K$-linear isomorphism determined by
$
\psi(e_i)=g_i
$
for $1\leq i\leq n$. Thus, if
$
x=\sum_{i=1}^n x_i e_i\in K^n,
$
then
$
\psi(x)=\sum_{i=1}^n x_i g_i\in KG.
$

Via this identification, additive codes in $KG$ inherit the usual notions of Hamming weight, minimum distance, and code parameters from additive codes in $K^n$. In the same way, the standard Euclidean, trace Euclidean, Hermitian, and trace Hermitian forms on $K^n$ give rise to the corresponding forms on $KG$, which we denote by the same symbols. More precisely, for $x=\sum_{g\in G} a_g g$ and $y=\sum_{g\in G} b_g g$ in $KG$, we define
$$
\langle x,y\rangle_{\mathrm{E}}=\sum_{g\in G} a_g b_g
\;\;\; \textrm{and} \;\;\;
\langle x,y\rangle_{\mathrm{TE}}=\operatorname{Tr}_{K/F}\!\left(\sum_{g\in G} a_g b_g\right),
$$
where the trace map from $K$ to $F$ is given by
$$
\operatorname{Tr}_{K/F}(x)=x+x^q+\cdots+x^{q^{m-1}}
$$
for all $x\in K$. This map is $F$-linear and surjective.

When $m$ is even, we also define
$$
\langle x,y\rangle_{\mathrm{H}}=\sum_{g\in G} a_g \overline{b_g} \;\;\; \textrm{and} \;\;\;
\langle x,y\rangle_{\mathrm{TH}}=\operatorname{Tr}_{K/F}\!\left(\sum_{g\in G} a_g \overline{b_g}\right),
$$
where, for
$
a=\sum_{g\in G} a_g g\in KG,
$
we set
$
\overline{a}:=\sum_{g\in G} \overline{a_g}\,g,
$
and $\overline{a_g}=a_g^{q^{m/2}}$ denotes the involution of $K$.

These are called the \emph{Euclidean}, \emph{trace Euclidean}, \emph{Hermitian}, and \emph{trace Hermitian} forms on $KG$, respectively.

\begin{remark}
The Euclidean form is a non-degenerate symmetric bilinear form on $KG$ over $K$, and the trace Euclidean form is a non-degenerate symmetric bilinear form on $KG$ over $F$. If $m$ is even, then the Hermitian form is a non-degenerate  Hermitian sesquilinear  form on $KG$ over $K$, whereas the trace Hermitian form is a non-degenerate symmetric bilinear form on $KG$ over $F$, since the involution on $K$ fixes $F$ pointwise.
\end{remark}

\begin{lemma}\label{identidad}
Let $a,x,y\in KG$. Then the following hold:
\begin{enumerate}
    \item The Euclidean form satisfies
    $
    \langle ax,y\rangle_{\mathrm{E}}=\langle x,a^\ast y\rangle_{\mathrm{E}}$ and $
    \langle xa,y\rangle_{\mathrm{E}}=\langle x,ya^\ast\rangle_{\mathrm{E}}.
    $
    Hence, the trace Euclidean form satisfies
    $
    \langle ax,y\rangle_{\mathrm{TE}}=\langle x,a^\ast y\rangle_{\mathrm{TE}}$ and $\langle xa,y\rangle_{\mathrm{TE}}=\langle x,ya^\ast\rangle_{\mathrm{TE}}.
    $

    \item If $m$ is even, then the Hermitian form satisfies
    $
    \langle ax,y\rangle_{\mathrm{H}}=\langle x,\overline{a}^\ast y\rangle_{\mathrm{H}}$ and $
    \langle xa,y\rangle_{\mathrm{H}}=\langle x,y\,\overline{a}^\ast\rangle_{\mathrm{H}}.
    $
    Hence, the trace Hermitian form satisfies
    $
    \langle ax,y\rangle_{\mathrm{TH}}=\langle x,\overline{a}^\ast y\rangle_{\mathrm{TH}}$ and $
    \langle xa,y\rangle_{\mathrm{TH}}=\langle x,y\,\overline{a}^\ast\rangle_{\mathrm{TH}}.
    $
\end{enumerate}
\end{lemma}

\begin{proof}
Write
$
a=\sum_{g\in G} a_g g,$ $
x=\sum_{h\in G} x_h h,$ and
$y=\sum_{k\in G} y_k k.$
Then we have
$$
\langle ax,y\rangle_{\mathrm{E}}
=
\sum_{g,h\in G} a_gx_h\,y_{gh}
=
\sum_{h\in G} x_h (\sum_{g\in G} a_g\,y_{gh}).
$$
Now $a^\ast y=\sum_{g,k\in G} a_g y_k\,g^{-1}k,
$
so the coefficient of $h$ in $a^\ast y$ is $\sum_{g\in G} a_g\,y_{gh}$. Hence
$$
\langle x,a^\ast y\rangle_{\mathrm{E}}
=
\sum_{h\in G} x_h (\sum_{g\in G} a_g\,y_{gh})
=
\langle ax,y\rangle_{\mathrm{E}}.
$$
Thus $\langle ax,y\rangle_{\mathrm{E}}=\langle x,a^\ast y\rangle_{\mathrm{E}}$. The identity
$
\langle xa,y\rangle_{\mathrm{E}}=\langle x,ya^\ast\rangle_{\mathrm{E}}
$
is analogous. Applying $\operatorname{Tr}_{K/F}$ gives the corresponding identities for $\langle\cdot,\cdot\rangle_{\mathrm{TE}}$.

Assume now that $m$ is even. Then
$$
\langle ax,y\rangle_{\mathrm{H}}
=
\sum_{g,h\in G} a_gx_h\,\overline{y_{gh}}
=
\sum_{h\in G} x_h \sum_{g\in G} a_g\,\overline{y_{gh}}.
$$
Also,
$
\overline{a}^\ast=\sum_{g\in G} \overline{a_g}\,g^{-1} $ and $
\overline{a}^\ast y=\sum_{g,k\in G} \overline{a_g}\,y_k\,g^{-1}k,
$
so the coefficient of $h$ in $\overline{a}^\ast y$ is $\sum_{g\in G} \overline{a_g}\,y_{gh}$. Therefore
$$
\langle x,\overline{a}^\ast y\rangle_{\mathrm{H}}
=
\sum_{h\in G} x_h\,
\overline{\sum_{g\in G} \overline{a_g}\,y_{gh}}
=
\sum_{h\in G} x_h \sum_{g\in G} a_g\,\overline{y_{gh}}
=
\langle ax,y\rangle_{\mathrm{H}}.
$$
Thus $\langle ax,y\rangle_{\mathrm{H}}=\langle x,\overline{a}^\ast y\rangle_{\mathrm{H}}$. The identity
$
\langle xa,y\rangle_{\mathrm{H}}=\langle x,y\,\overline{a}^\ast\rangle_{\mathrm{H}}
$
is analogous. Applying $\operatorname{Tr}_{K/F}$ gives the corresponding identities for $\langle\cdot,\cdot\rangle_{\mathrm{TH}}$.
\end{proof}

For $a=\sum_{g\in G} a_g g\in KG$, we denote by $\operatorname{coef}_{1_G}(a)$ the coefficient of $1_G$ in $a$.

\begin{proposition}
Let $x,y\in KG$. Then the following hold:
\begin{enumerate}
    \item The Euclidean and trace Euclidean forms are given by
    $
    \langle x,y\rangle_{\mathrm{E}}
    =
    \operatorname{coef}_{1_G}(xy^\ast)$ and
$ \langle x,y\rangle_{\mathrm{TE}}
    =   \operatorname{Tr}_{K/F}\!\bigl(\operatorname{coef}_{1_G}(xy^\ast)\bigr).
    $
    \item If $m$ is even, then the Hermitian and trace Hermitian forms are given by
    $    \langle x,y\rangle_{\mathrm{H}}
    =
    \operatorname{coef}_{1_G}(x\overline{y}^\ast)$ and $
    \langle x,y\rangle_{\mathrm{TH}}
    =
    \operatorname{Tr}_{K/F}\!\bigl(\operatorname{coef}_{1_G}(x\overline{y}^\ast)\bigr).
    $
\end{enumerate}
\end{proposition}

\begin{proof}
Write
$
x=\sum_{g\in G} x_g g,$ and $
y=\sum_{h\in G} y_h h.
$

\begin{enumerate}
    \item Since
    $
    y^\ast=\sum_{h\in G} y_h h^{-1},
    $
    we have
    $
    xy^\ast
    =
    \sum_{g,h\in G} x_g y_h\,gh^{-1}.
    $
    The coefficient of $1_G$ is obtained precisely from the terms with $g=h$. Hence, we obtain
    $
    \operatorname{coef}_{1_G}(xy^\ast)
    =
    \sum_{g\in G} x_g y_g
    =
    \langle x,y\rangle_{\mathrm{E}}.
    $
    Applying $\operatorname{Tr}_{K/F}$ yields
    $
    \langle x,y\rangle_{\mathrm{TE}}
    =
    \operatorname{Tr}_{K/F}\!\bigl(\operatorname{coef}_{1_G}(xy^\ast)\bigr).
    $

    \item Assume that $m$ is even. Since
    $
    \overline{y}^\ast
    =
    \sum_{h\in G} \overline{y_h}\,h^{-1},
    $
    it follows that, similarly,
    $
    x\overline{y}^\ast
    =
    \sum_{g,h\in G} x_g\overline{y_h}\,gh^{-1}.
    $
    Again, the coefficient of $1_G$ is obtained precisely from the terms with $g=h$. Therefore
    $
    \operatorname{coef}_{1_G}(x\overline{y}^\ast)
    =
    \sum_{g\in G} x_g\overline{y_g}
    =
    \langle x,y\rangle_{\mathrm{H}}.
    $
    Applying $\operatorname{Tr}_{K/F}$ yields
    $
    \langle x,y\rangle_{\mathrm{TH}}
    =
    \operatorname{Tr}_{K/F}\!\bigl(\operatorname{coef}_{1_G}(x\overline{y}^\ast)\bigr).
    $
\end{enumerate}
\end{proof}

\begin{definition}
Let $C \leq KG$ be an additive left group code in $KG$. The duals of $C$ with respect to the Euclidean and trace-Euclidean forms are denoted by $C^{\perp_{\mathrm{E}}}$ and $C^{\perp_{\mathrm{TE}}}$, respectively. If $m$ is even, the duals of $C$ with respect to the Hermitian and trace-Hermitian forms are denoted by $C^{\perp_{\mathrm{H}}}$ and $C^{\perp_{\mathrm{TH}}}$, respectively.
\end{definition}

\begin{proposition}
Let $C \leq KG$ be an additive left group code in $KG$. Then $C^{\perp_{\mathrm{E}}}$ and $C^{\perp_{\mathrm{TE}}}$ are additive left group codes in $KG$. If $m$ is even, then $C^{\perp_{\mathrm{H}}}$ and $C^{\perp_{\mathrm{TH}}}$ are also additive left group codes in $KG$.
\end{proposition}
\begin{proof}
Let $C^{\perp_\star}$ denote any of the duals with respect to the forms under consideration. By definition,
$$
C^{\perp_\star}=\{x\in KG:\langle x,c\rangle_\star=0 \text{ for all } c\in C\}.
$$
Clearly, $C^{\perp_\star}$ is an additive subgroup of $KG$. We only need to show that it is closed under left multiplication by elements of $FG$.

Let $a\in FG$ and $x\in C^{\perp_\star}$. For any $c\in C$, Lemma~\ref{identidad} gives
$$
\langle ax,c\rangle_\star=\langle x,\theta(a)c\rangle_\star,
$$
where $\theta(a)=a^\ast$ for the Euclidean and trace Euclidean forms, and $\theta(a)=\overline{a}^\ast$ for the Hermitian and trace Hermitian forms. In each case, $\theta(a)\in FG$. Since $C$ is an additive left group code in $KG$, it follows that $\theta(a)c\in C$. Hence,
$$
\langle x,\theta(a)c\rangle_\star=0,
$$
because $x\in C^{\perp_\star}$. Therefore $\langle ax,c\rangle_\star=0$ for all $c\in C$, which shows that $ax\in C^{\perp_\star}$.
\end{proof}
\begin{definition}
Let $C \leq KG$ be an additive left group code. For $\star\in\{\mathrm{E},\mathrm{TE}\}$, we say that $C$ is \emph{$\star$-LCD} if $C\cap C^{\perp_{\star}}=\{0\}$, and \emph{$\star$-self-dual} if $C=C^{\perp_{\star}}$. If $m$ is even, the same terminology is used for $\star\in\{\mathrm{H},\mathrm{TH}\}$.
\end{definition}

Since the trace Euclidean form is non-degenerate, every additive group code $C \leq KG$ satisfies
$$
\dim_F(C)+\dim_F(C^{\perp_{\mathrm{TE}}})=mn.
$$
If $m$ is even, then the trace Hermitian form is also non-degenerate, and hence every additive code $C \leq KG$ satisfies
$
\dim_F(C)+\dim_F(C^{\perp_{\mathrm{TH}}})=mn.
$
In particular,
$
\left(C^{\perp_{\mathrm{TE}}}\right)^{\perp_{\mathrm{TE}}}=C$ and
$\left(C^{\perp_{\mathrm{TH}}}\right)^{\perp_{\mathrm{TH}}}=C.
$

\begin{remark}
For $\star=\mathrm{TE}$, an additive left group code $C \leq KG$ is $\star$-LCD if and only if
$$
C\oplus C^{\perp_{\mathrm{TE}}}=KG
$$
as a direct sum of left $FG$-submodules, since $\langle\cdot,\cdot\rangle_{\mathrm{TE}}$ is a non-degenerate $F$-bilinear form on $KG$. If $m$ is even, the same statement holds for $\star=\mathrm{TH}$.
\end{remark}

\begin{proposition}
\label{proposition:Ksubspace-orthogonals}
Let $C\leq KG$ be a left ideal, that is, a left group code in $KG$. Then
$
C^{\perp_{\mathrm{E}}}=C^{\perp_{\mathrm{TE}}}.
$
If $m$ is even, then also
$
C^{\perp_{\mathrm{H}}}=C^{\perp_{\mathrm{TH}}}.
$
\end{proposition}

\begin{proof}
Since $C$ is a left ideal of $KG$, it is in particular a $K$-subspace of $KG$. We prove the Euclidean case; the Hermitian case is analogous.

Because
$
\langle x,y\rangle_{\mathrm{TE}}
=
\operatorname{Tr}_{K/F}\bigl(\langle x,y\rangle_{\mathrm{E}}\bigr),
$
it follows immediately that
$
C^{\perp_{\mathrm{E}}}\subseteq C^{\perp_{\mathrm{TE}}}.
$

For the reverse inclusion, let $y\notin C^{\perp_{\mathrm{E}}}$. Then there exists $x\in C$ such that
$
a:=\langle x,y\rangle_{\mathrm{E}}\neq 0.
$
Since $C$ is a $K$-subspace, we have $\lambda x\in C$ for every $\lambda\in K$. By the non-degeneracy of the trace map, there exists $\lambda\in K$ such that
$
\operatorname{Tr}_{K/F}(\lambda a)\neq 0.
$
Therefore,
$$
\langle \lambda x,y\rangle_{\mathrm{TE}}
=
\operatorname{Tr}_{K/F}\bigl(\langle \lambda x,y\rangle_{\mathrm{E}}\bigr)
=
\operatorname{Tr}_{K/F}(\lambda \langle x,y\rangle_{\mathrm{E}})
=
\operatorname{Tr}_{K/F}(\lambda a)\neq 0.
$$
Hence $y\notin C^{\perp_{\mathrm{TE}}}$, and so
$
C^{\perp_{\mathrm{TE}}}\subseteq C^{\perp_{\mathrm{E}}}.
$
We conclude that
$
C^{\perp_{\mathrm{E}}}=C^{\perp_{\mathrm{TE}}}.
$

If $m$ is even, the Hermitian case is proved in exactly the same way, using
$
\langle x,y\rangle_{\mathrm{TH}}
=
\operatorname{Tr}_{K/F}\bigl(\langle x,y\rangle_{\mathrm{H}}\bigr).
$
\end{proof}

\begin{remark}\label{hexa-code}
Recall that the Hexacode $\mathcal{H}_6$ is a Hermitian self-dual
$\mathbb{F}_4$-linear code with parameters $[6,3,4]$, and hence also a
trace-Hermitian self-dual additive code with parameters $(6,2^6,4)$. By
\cite{BBW24}, $\mathcal{H}_6$ is not an $\mathbb{F}_4$-linear group code for
any group of order $6$.

We claim that $\mathcal{H}_6$ is not an additive left group code for any
group of order $6$. Indeed, let $G\in\{C_6,S_3\}$ and suppose, for a
contradiction, that $\mathcal{H}_6$ is an additive left group code in
$\mathbb{F}_4G$. Then $\mathcal{H}_6$ is stable under left multiplication by
$\mathbb{F}_2G$. Since $\mathcal{H}_6$ is also $\mathbb{F}_4$-linear, it is
stable under multiplication by the scalars in $\mathbb{F}_4$. Therefore, for
$a=\sum_{g\in G}\lambda_g g\in\mathbb{F}_4G$ and
$c\in\mathcal{H}_6$, one has
$ac=\sum_{g\in G}\lambda_g(gc)\in\mathcal{H}_6$. Thus,
$\mathcal{H}_6$ is stable under left multiplication by $\mathbb{F}_4G$, and
hence it is an $\mathbb{F}_4$-linear left group code. This contradicts
\cite{BBW24}.

More generally, let $C\leq K^n$ be a $K$-linear code. If $C$ is not a
$K$-linear left group code for a group $G$, then $C$ cannot be an additive
left group code for $G$.
\end{remark}
\section{Projectors for additive group codes}
An important class of group codes in $KG$ is that of idempotent codes. These
are the group codes of the form $C=KGe$, where $e\in KG$ is an idempotent,
that is, $e^2=e$. In this case, $C$ is the image of the map
$\rho_e:KG\longrightarrow KG$ defined by $\rho_e(x)=xe$. Since
$\rho_e^2=\rho_e$, the map $\rho_e$ is a $KG$-linear projector, and
$C=\operatorname{Im}(\rho_e)$ is a left ideal of $KG$.

Conversely, every $KG$-linear projector on $KG$ is induced by right
multiplication by an idempotent of $KG$. Indeed, let
$P\in\operatorname{End}_{KG}(KG)$ satisfy $P^2=P$, and set $e=P(1)$. Then,
for every $x\in KG$, one has $P(x)=xP(1)=xe$. Moreover,
$e^2=P(e)=P^2(1)=P(1)=e$. Hence $e$ is an idempotent and
$\operatorname{Im}(P)=KGe$. Thus, the idempotent group codes are precisely
the codes that arise as images of $KG$-linear projectors on $KG$.

However, this classical construction is too restrictive in the additive
setting. Indeed, additive left group codes are only required to be left
$FG$-submodules of $KG$ and need not be $K$-linear, so they do not
necessarily arise as left ideals of $KG$. This naturally leads to the more
general framework of $FG$-linear projectors on $KG$ and their images.

\subsection{$FG$-projectors}
\begin{definition}
An element $P\in \operatorname{End}_{FG}(KG)$ is called an \emph{$FG$-linear projector} if
$
P^2=P.
$

\end{definition}
Note that if $P\in \operatorname{End}_{FG}(KG)$ is a projector, then $\operatorname{Im}(P)$ and $\ker(P)$ are left $FG$-submodules of $KG$, that is, left group codes in $KG$, and
$
KG=\operatorname{Im}(P)\oplus\ker(P)
$
as a direct sum of $FG$-modules.

\begin{definition}
Let $B$ be a non-degenerate $F$-bilinear form on $KG$, and let $T\in \operatorname{End}_{FG}(KG)$. The \emph{adjoint} of $T$ with respect to $B$ is the unique map $T^{\ast_B}\in \operatorname{End}_{F}(KG)$ such that $B(Tx,y)=B(x,T^{\ast_B} y)$ for all $x,y\in KG$.
\end{definition}
\begin{remark}\label{adjunto-E-TE}
Let $T\in \operatorname{End}_{FG}(KG)$. The adjoint of $T$ with respect to the form $\langle\cdot,\cdot\rangle_\star$ will be denoted by $T^{\ast_\star}$. However, whenever no confusion is likely to arise, we shall simply write $T^\ast$.
\end{remark}
Since $KG$ is finite-dimensional over $F$, every non-degenerate $F$-bilinear form on $KG$ induces an $F$-linear isomorphism
$
KG \longrightarrow \operatorname{Hom}_F(KG,F), \;y\longmapsto \bigl(x\mapsto B(x,y)\bigr).
$
Hence, if $B$ is a non-degenerate $F$-bilinear form on $KG$ and $T\in \operatorname{End}_{FG}(KG)$, then for each $y\in KG$ the map $x\mapsto B(Tx,y)$ is an $F$-linear functional on $KG$, so there exists a unique element $T^\ast y\in KG$ such that
$
B(Tx,y)=B(x,T^\ast y)\text{ for all }x\in KG.
$
In particular, every $FG$-linear endomorphism of $KG$ admits a unique adjoint with respect to $\langle\cdot,\cdot\rangle_{\mathrm{TE}}$ and, when $m$ is even, with respect to $\langle\cdot,\cdot\rangle_{\mathrm{TH}}$.

Note that the same argument does not apply to $\langle\cdot,\cdot\rangle_{\mathrm{E}}$ or, when $m$ is even, to $\langle\cdot,\cdot\rangle_{\mathrm{H}}$, since these are $K$-valued forms. If $T$ is merely $FG$-linear, then in general it need not be $K$-linear, and therefore the map $x\mapsto \langle T(x),y\rangle_\star$ need not be $K$-linear. Consequently, the adjoint of $T$ with respect to $\langle\cdot,\cdot\rangle_{\mathrm{E}}$ or $\langle\cdot,\cdot\rangle_{\mathrm{H}}$ need not exist as a $K$-linear operator.

On the other hand, if $T$ is $K$-linear, then its adjoint with respect to $\langle\cdot,\cdot\rangle_{\mathrm{E}}$ exists, and one has $T^{\ast_{\mathrm{TE}}}=T^{\ast_{\mathrm{E}}}$. Indeed, if
$$
\langle Tx,y\rangle_{\mathrm{E}}=\langle x,T^{\ast_{\mathrm{E}}}y\rangle_{\mathrm{E}}\;
\text{ for all }x,y\in KG,
$$
then applying the trace map $\operatorname{Tr}_{K/F}$ to both sides gives
$$
\langle Tx,y\rangle_{\mathrm{TE}}=\langle x,T^{\ast_{\mathrm{E}}}y\rangle_{\mathrm{TE}}\;
\text{ for all }x,y\in KG.
$$
By uniqueness of the adjoint with respect to $\langle\cdot,\cdot\rangle_{\mathrm{TE}}$, it follows that $T^{\ast_{\mathrm{TE}}}=T^{\ast_{\mathrm{E}}}$. Similarly, when $m$ is even and $T$ is $K$-linear, one has $T^{\ast_{\mathrm{TH}}}=T^{\ast_{\mathrm{H}}}$.

\begin{lemma}\label{lem:adjoint-FG-linear}
Let $\star\in\{\mathrm{TE},\mathrm{TH}\}$, where the case $\star=\mathrm{TH}$ is considered only when $m$ is even. If $T\in \operatorname{End}_{FG}(KG)$, then its adjoint $T^\ast$ with respect to $\langle\cdot,\cdot\rangle_\star$ also belongs to $\operatorname{End}_{FG}(KG)$.
\end{lemma}

\begin{proof}
We treat the case $\star=\mathrm{TE}$; the case $\star=\mathrm{TH}$ is analogous. Let $a\in FG$ and $x,y\in KG$. Since $T$ is $FG$-linear, one has $T(a^\ast x)=a^\ast T(x)$. Using Lemma~\ref{identidad}, we obtain

\begin{align*}
\langle x,T^\ast(ay)\rangle_{\mathrm{TE}}
&= \langle T(x),ay\rangle_{\mathrm{TE}}
= \langle a^\ast T(x),y\rangle_{\mathrm{TE}}
= \langle T(a^\ast x),y\rangle_{\mathrm{TE}} \\
&= \langle a^\ast x,T^\ast y\rangle_{\mathrm{TE}}
= \langle x,aT^\ast y\rangle_{\mathrm{TE}}.
\end{align*}
Since $\langle\cdot,\cdot\rangle_{\mathrm{TE}}$ is non-degenerate, it follows that $T^\ast(ay)=aT^\ast(y)$ for all $a\in FG$ and $y\in KG$. Thus $T^\ast\in \operatorname{End}_{FG}(KG)$.

If $m$ is even, the same argument applies to $\star=\mathrm{TH}$, using Lemma~\ref{identidad} with $\overline{a}^\ast$ in place of $a^\ast$.
\end{proof}

\begin{lemma}\label{lem:adjoint-basic}
Let $\langle\cdot,\cdot\rangle$ be a non-degenerate $F$-bilinear form on $KG$, and let $T,S\in \operatorname{End}_{FG}(KG)$. Then the following hold:
\begin{enumerate}
    \item $(T^\ast)^\ast=T$.
    \item $(TS)^\ast=S^\ast T^\ast$.
    \item $(\operatorname{Im}T)^\perp=\ker(T^\ast)$.
    \item $(\ker T)^\perp=\operatorname{Im}(T^\ast)$.
\end{enumerate}
Here $T^\ast$ denotes the adjoint of $T$ with respect to $\langle\cdot,\cdot\rangle$.
\end{lemma}

\begin{proof}
\begin{enumerate}
    \item For all $x,y\in KG$, we have $\langle Tx,y\rangle=\langle x,T^\ast y\rangle=\langle (T^\ast)^\ast x,y\rangle$. Since $\langle\cdot,\cdot\rangle$ is non-degenerate, it follows that $(T^\ast)^\ast=T$.

    \item For all $x,y\in KG$, we have
    $\langle TSx,y\rangle=\langle Sx,T^\ast y\rangle=\langle x,S^\ast T^\ast y\rangle$.
    Hence $(TS)^\ast=S^\ast T^\ast$.

    \item Let $y\in KG$. Then $y\in (\operatorname{Im}T)^\perp$ if and only if $\langle Tx,y\rangle=0$ for all $x\in KG$, which is equivalent to $\langle x,T^\ast y\rangle=0$ for all $x\in KG$. Since $\langle\cdot,\cdot\rangle$ is non-degenerate, this holds if and only if $T^\ast y=0$. Hence $(\operatorname{Im}T)^\perp=\ker(T^\ast)$.

    \item Applying part~(3) to $T^\ast$, we get $(\operatorname{Im}T^\ast)^\perp=\ker((T^\ast)^\ast)$. By part~(1), $(T^\ast)^\ast=T$. Therefore $(\operatorname{Im}T^\ast)^\perp=\ker(T)$. Taking orthogonals again gives $(\ker T)^\perp=\operatorname{Im}(T^\ast)$.
\end{enumerate}
\end{proof}

The projector approach to additive left group codes is naturally connected with the notions of direct summands and projective modules over $FG$. Recall that a left $FG$-module $M$ is called \textit{projective} if and only if it is a direct summand of a free left $FG$-module.

The next result shows that additive left group codes arising as images of $FG$-linear projectors are precisely the direct summands of $KG$, and hence form a natural subclass of projective left $FG$-modules.

\begin{proposition}\label{proj-summand}
Let $C\le KG$ be an additive left group code, equivalently, a left $FG$-submodule of $KG$. Then the following are equivalent:
\begin{enumerate}
    \item There exists an $FG$-linear projector $P\in\operatorname{End}_{FG}(KG)$ such that $C=\operatorname{Im}(P)$.
    \item $C$ is a direct summand of $KG$ as a left $FG$-module.
\end{enumerate}
\end{proposition}

\begin{proof}
Assume first that $C=\operatorname{Im}(P)$ for some $FG$-linear projector $P$ on $KG$. Since $P^2=P$, we have $KG=\operatorname{Im}(P)\oplus\ker(P)$ as a direct sum of left $FG$-modules. Hence $C$ is a direct summand of $KG$.

Conversely, suppose that $C$ is a direct summand of $KG$ as a left $FG$-module. Then there exists a left $FG$-submodule $C'\le KG$ such that $KG=C\oplus C'$. Define $P:KG\to KG$ by $P(c+c')=c$ for all $c\in C$ and $c'\in C'$. Then $P$ is well defined, $FG$-linear, and satisfies $P^2=P$. Moreover, $\operatorname{Im}(P)=C$.
\end{proof}

\begin{definition}\label{def:projector-additive-code}
An additive left group code $C\le KG$ is called a \textit{projector additive left group code} if there exists an $FG$-linear projector $P\in\operatorname{End}_{FG}(KG)$ such that
$
C=\operatorname{Im}(P).
$

\end{definition}

\begin{remark}\label{projective}
By Proposition~\ref{proj-summand}, an additive left group code $C \leq KG$ is a projector additive left group code if and only if it is a direct summand of $KG$ as a left $FG$-module.
Moreover, since $KG$ is a free left $FG$-module, every projector additive left group code is projective as a left $FG$-module.
\end{remark}
\begin{corollary}\label{cor:semisimple-projector}
Assume that $FG$ is semisimple; for instance, this holds whenever $\operatorname{char}(F)\nmid |G|$. Let $C\subseteq KG$. Then the following conditions are equivalent:
\begin{enumerate}
    \item $C$ is an additive left group code in $KG$.
    \item $C$ is a direct summand of $KG$ as a left $FG$-module.
    \item $C$ is a projector additive left group code.
    \item $C$ is a projective left $FG$-submodule of $KG$.
\end{enumerate}
\end{corollary}

\begin{proof}
Since $FG$ is semisimple, every left $FG$-module is semisimple. In particular, $KG$ is semisimple as a left $FG$-module, and therefore every left $FG$-submodule of $KG$ is a direct summand of $KG$. Thus $(1)\Leftrightarrow(2)$.

The equivalence $(2)\Leftrightarrow(3)$ follows from Proposition~\ref{proj-summand}, and $(3)\Rightarrow(4)$ follows from Remark~\ref{projective}. Finally, $(4)\Rightarrow(1)$ is immediate.
\end{proof}

We now show that the classical ideal-theoretic construction is recovered as a special case of the projector approach.

\begin{proposition}\label{idempotent-projector}
Let $e\in KG$ satisfy $e^2=e$. Then the following hold:
\begin{enumerate}
    \item Right multiplication by $e$ defines an $FG$-linear projector $\rho_e:KG\longrightarrow KG$ by $\rho_e(x)=xe$.
    \item $
\operatorname{Im}(\rho_e)=KGe$ and $
\ker(\rho_e)=KG(1-e).
$ Consequently,
$
KG=KGe\oplus KG(1-e)
$
as a direct sum of left $FG$-modules.
\item $
\rho_e(FG)=FGe$ is an additive left group code.
\end{enumerate}
\end{proposition}

\begin{proof} \begin{enumerate}
\item Since $e^2=e$, for every $x\in KG$ we have
$$
\rho_e^2(x)=\rho_e(xe)=xee=xe=\rho_e(x).
$$
Thus $\rho_e$ is a projector. It is also $FG$-linear, because for all $a\in FG$ and $x\in KG$,
$$
\rho_e(ax)=(ax)e=a(xe)=a\rho_e(x).
$$
\item Clearly, $\operatorname{Im}(\rho_e)=KGe$. Now let $x\in KG$. If $x\in KG(1-e)$, say $x=y(1-e)$ for some $y\in KG$, then
$
\rho_e(x)=y(1-e)e=y(e-e^2)=0.
$
Hence $KG(1-e)\subseteq \ker(\rho_e)$. Conversely, if $x\in \ker(\rho_e)$, then $xe=0$, and therefore
$
x=x-xe=x(1-e).
$
Thus $x\in KG(1-e)$, and so $\ker(\rho_e)=KG(1-e)$.

Since $\rho_e$ is a projector, we obtain
$
KG=\operatorname{Im}(\rho_e)\oplus\ker(\rho_e)=KGe\oplus KG(1-e).
$
In particular, Proposition~\ref{proj-summand} shows that $KGe$ is a direct summand of the left $FG$-module $KG$.
\item Finally, $\rho_e(FG)=FGe$ is a left $FG$-submodule of $KG$, and therefore an additive left group code.
\end{enumerate}
\end{proof}

\begin{remark}\label{propiedades-adjunto}
\begin{enumerate}
\item Let $e\in KG$ be an idempotent. Then $e^\ast$ is also an idempotent, since $(e^\ast)^2=(e^2)^\ast=e^\ast$. Consequently, right multiplication by $e^\ast$ defines an $FG$-linear projector $\rho_{e^\ast}:KG\longrightarrow KG$, given by $\rho_{e^\ast}(x)=xe^\ast$.

If $\rho_e^{\ast_{\mathrm E}}$ and $\rho_e^{\ast_{\mathrm{TE}}}$ denote the adjoints of $\rho_e$ with respect to $\langle\cdot,\cdot\rangle_{\mathrm E}$ and $\langle\cdot,\cdot\rangle_{\mathrm{TE}}$, respectively, then $\rho_e^{\ast_{\mathrm E}}=\rho_{e^\ast}$ and $\rho_e^{\ast_{\mathrm{TE}}}=\rho_{e^\ast}$. In particular, $\rho_e^{\ast_{\mathrm E}}=\rho_e^{\ast_{\mathrm{TE}}}$.

If $m$ is even and $\rho_e^{\ast_{\mathrm H}}$ and $\rho_e^{\ast_{\mathrm{TH}}}$ denote the adjoints of $\rho_e$ with respect to $\langle\cdot,\cdot\rangle_{\mathrm H}$ and $\langle\cdot,\cdot\rangle_{\mathrm{TH}}$, respectively, then $\rho_e^{\ast_{\mathrm H}}=\rho_{\overline e^\ast}$ and $\rho_e^{\ast_{\mathrm{TH}}}=\rho_{\overline e^\ast}$. In particular, $\rho_e^{\ast_{\mathrm H}}=\rho_e^{\ast_{\mathrm{TH}}}$.

\item Using Proposition~\ref{idempotent-projector} and  Lemma~\ref{lem:adjoint-basic}, a direct computation for the Euclidean and trace-Euclidean cases yields:
\begin{enumerate}
    \item $\operatorname{Im}(\rho_e) = \ker\big(\rho_e^{\ast_{\mathrm E}}\big)^{\perp_{\mathrm E}} = KG\,e.$
    \item $\ker(\rho_e) = \operatorname{Im}\big(\rho_e^{\ast_{\mathrm E}}\big)^{\perp_{\mathrm E}} = KG\,(1 - e).$
    \item $\operatorname{Im}\big(\rho_e^{\ast_{\mathrm E}}\big) = \ker(\rho_e)^{\perp_{\mathrm E}} = KG\,e^{\ast}.$
    \item $\ker\big(\rho_e^{\ast_{\mathrm E}}\big) = \operatorname{Im}(\rho_e)^{\perp_{\mathrm E}} = KG\,(1 - e^{\ast}).$
\end{enumerate}

Analogous identities hold for the Hermitian and trace-Hermitian forms.
\end{enumerate}
\end{remark}

\begin{definition}\label{def:restricted-idempotent}
Let $e\in KG$ be an idempotent, and let
$\rho_e:KG\longrightarrow KG$ be the map defined by $\rho_e(x)=xe$. The
code $\operatorname{Im}(\rho_e)=KGe$ is the classical
\textit{idempotent group code} associated with $e$. The additive left group
code $\rho_e(FG)=FGe$ is called the \textit{restricted idempotent additive
left group code} associated with $e$, or simply, by a slight abuse of
terminology, the \textit{idempotent additive left group code} associated with
$e$.
\end{definition}

\begin{example}\label{ex:restricted-idempotent-basic}
Let $F=\mathbb{F}_2$, let $K=\mathbb{F}_8$, and let
$G=C_3=\langle g\rangle$ with $g^3=1$. Consider the idempotent
$e=e^2=1+g+g^2\in FG\subseteq KG$. Every element of $FG$ has the form $a+bg+cg^2$, where
$a,b,c\in\mathbb{F}_2$, and multiplication by $e$ gives
$(a+bg+cg^2)e=(a+b+c)(1+g+g^2)$. Hence $FGe=\{0,e\}$. Therefore, $FGe$ is a
restricted idempotent additive left group code with parameters $(3,2,3)$ over
$\mathbb{F}_8$. Moreover, $FGe$ is $\mathrm{TE}$-LCD. Since $e$ is its only nonzero
codeword, it suffices to compute its trace-Euclidean self-inner product. The
coordinates of $e$ are $(1,1,1)$, and therefore
$\langle e,e\rangle_{\mathrm{TE}}=\operatorname{Tr}_{K/F}(1+1+1)
=\operatorname{Tr}_{K/F}(1)=1\neq 0$, since $[K:F]=3$. Hence
$FGe\cap(FGe)^{\perp_{\mathrm{TE}}}=\{0\}$. The dual
$(FGe)^{\perp_{\mathrm{TE}}}$ has parameters $(3,2^8,1)$ and is not a
restricted idempotent additive left group code. Indeed, every code of the
form $FGf$, where $f\in KG$ is an idempotent, has $F$-dimension at most
$\dim_F(FG)=3$, whereas
$\dim_F((FGe)^{\perp_{\mathrm{TE}}})=8$.
\end{example}

\begin{remark}\label{rem:projector-vs-restricted-idempotent}
Unlike the classical $K$-linear setting, restricted idempotent additive group
codes and projector additive group codes do not, in general, coincide.
Recall that, in the classical setting, the codes arising as images of
$KG$-linear projectors are precisely the idempotent group codes. In the
additive setting, a restricted idempotent additive group code has the form
$FGe$ for an idempotent $e\in KG$, whereas a projector additive group code is
the image of an arbitrary $FG$-linear projector on $KG$.

Assume that $FG$ is semisimple. Then, by
Corollary~\ref{cor:semisimple-projector}, every additive left group code in
$KG$ is a projector additive left group code. In particular, every restricted
idempotent additive group code $FGe$ is a projector additive left group code.
Thus, there exists an $FG$-linear projector $P:KG\longrightarrow KG$ such
that $\operatorname{Im}(P)=FGe$.

The projector $P$ need not be the map
$\rho_e:KG\longrightarrow KG$ given by $\rho_e(x)=xe$. Indeed,
$\operatorname{Im}(\rho_e)=KGe$, whereas the restricted idempotent additive
group code is $FGe=\rho_e(FG)$. The projector $P$ with image $FGe$ is instead
obtained by choosing a left $FG$-submodule $D\leq KG$ such that
$KG=FGe\oplus D$ and taking the projection onto $FGe$ along $D$.

Let $\mathcal{I}_{\mathrm{res}}$ denote the family of restricted idempotent
additive group codes, let $\mathcal{P}$ denote the family of projector
additive group codes, and let $\mathcal{A}$ denote the family of additive
left group codes in $KG$. Then, in the semisimple case,
$\mathcal{I}_{\mathrm{res}}\subseteq\mathcal{P}=\mathcal{A}$. The first
inclusion can be proper, as illustrated by
Example~\ref{ex:semisimple-projector}.

If $FG$ is not semisimple, then not every additive left group code is a
projector additive left group code. Moreover, a projector additive group code
need not be of the form $FGe$ for any idempotent $e\in KG$.

Examples~\ref{ex:semisimple-projector} and
\ref{ex:nonsemisimple-projector} illustrate the semisimple and
non-semisimple cases, respectively. Moreover, they show that projector additive group
codes can have parameters that are not attained by restricted idempotent
additive group codes. Thus, the projector construction provides not only a
flexible algebraic framework for additive group codes, but also a source of
genuinely additive codes with improved coding-theoretic parameters.
\end{remark}

\begin{example}\label{ex:semisimple-projector}
Let $F=\mathbb{F}_2$, let
$K=\mathbb{F}_4=\mathbb{F}_2(\omega)$, where
$\omega^2+\omega+1=0$, and let
$G=C_7=\langle g\mid g^7=1\rangle$. Since
$\operatorname{char}(F)=2\nmid 7=|G|$, the algebra $FG$ is semisimple.
Hence, by Corollary~\ref{cor:semisimple-projector}, every additive left group
code in $KG$ is a projector additive left group code.

Consider the elements
$b_1=1+(\omega+1)g^2+g^3+(\omega+1)g^4+\omega g^5+\omega g^6$,
$b_2=\omega+\omega g+g^2+(\omega+1)g^4+g^5+(\omega+1)g^6$, and
$b_3=(\omega+1)g+g^2+(\omega+1)g^3+\omega g^4+\omega g^5+g^6$.
Define $C=\langle b_1,b_2,b_3\rangle_{\mathbb{F}_2}\subseteq KG$.

A direct computation gives $gb_1=b_2+b_3$, $gb_2=b_1+b_2$, and $gb_3=b_1$.
Therefore $C$ is a left $FG$-submodule of $KG$, and hence, by
Corollary~\ref{cor:semisimple-projector}, $C$ is a projector additive left
group code. The elements $b_1,b_2,b_3$ are linearly independent
over $\mathbb{F}_2$, so $\dim_{\mathbb{F}_2}(C)=3$ and $|C|=2^3$. The
Hamming weight enumerator of $C$ is $1+7z^6$. Thus, $C$ has parameters
$(7,2^3,6)$ over $\mathbb{F}_4$.
The code $C$ attains the integral Singleton bound, since
$6=7-\left\lceil 3/2\right\rceil+1$. Hence, $C$ is Singleton-optimal and,
since $2\nmid 3$, it is a QMDS code rather than an MDS code in the strict
sense. On the other hand, a direct computation shows that $KC_7$ has eight
idempotents and that, among the restricted idempotent additive group codes
$FGe$ of $\mathbb{F}_2$-dimension $3$, there are precisely two distinct
codes, both having minimum distance $4$. Consequently,
$d(C)=6>\max\{d(FGe):e^2=e,\ \dim_{\mathbb{F}_2}(FGe)=3\}=4$. In particular,
$C$ is not a restricted idempotent additive group code.
\end{example}

\begin{example}\label{ex:nonsemisimple-projector}
Let $F=\mathbb{F}_2$, let
$K=\mathbb{F}_4=\mathbb{F}_2(\omega)$, where
$\omega^2+\omega+1=0$, and let
$G=C_4=\langle g\mid g^4=1\rangle$. Since
$\operatorname{char}(F)=2$ divides $|G|=4$, the group algebra
$FG=\mathbb{F}_2C_4$ is not semisimple. Set $R=FG$ and let
$u=1+g+g^2\in R$.

Consider the additive left group code
$C=R(1+\omega u)\subseteq KG$. Since $\omega$ is central in $KG$, one has
$x(1+\omega u)=x+\omega xu$ for every $x\in R$. Consequently,
$C=\{x+\omega xu:x\in R\}$.

Since $K=F\oplus\omega F$ as $F$-vector spaces, one has
$KG=R\oplus\omega R$ as left $R$-modules. Hence every element of $KG$ has a
unique expression $x+\omega y$, where $x,y\in R$. Define
$P:KG\longrightarrow KG$ by $P(x+\omega y)=x+\omega xu$. Let $a\in R$.
Then $P(a(x+\omega y))=P(ax+\omega ay)=ax+\omega(ax)u=aP(x+\omega y)$.
Hence $P$ is $R$-linear, equivalently, $FG$-linear. Moreover,
$P^2(x+\omega y)=P(x+\omega xu)=x+\omega xu=P(x+\omega y)$. Therefore,
$P^2=P$. Finally,
$\operatorname{Im}(P)=\{x+\omega xu:x\in R\}=C$. Thus $C$ is a projector
additive left group code.

Consider the $F$-linear map $\varphi:R\longrightarrow C$ defined by
$\varphi(x)=x+\omega xu$. The map $\varphi$ is surjective by the definition
of $C$. If $\varphi(x)=0$, then $x+\omega xu=0$. Since
$KG=R\oplus\omega R$, this implies $x=0$. Hence $\varphi$ is an
isomorphism. Therefore,
$\dim_{\mathbb{F}_2}(C)=\dim_{\mathbb{F}_2}(R)=4$, and $|C|=2^4$.

A direct computation gives the Hamming weight enumerator of $C$ as
$W_C(z)=1+12z^3+3z^4$. Hence $d(C)=3$, and $C$ has parameters
$(4,2^4,3)$ as an additive code over $\mathbb{F}_4$.

We claim that $C$ is not a restricted idempotent additive group code. Since
$KC_4\cong K[x]/(x^4-1)$ and the characteristic of $K$ is $2$, one has
$x^4-1=(x+1)^4$. Hence $KC_4\cong K[x]/(x+1)^4$ is a local ring. Its only
idempotents are $0$ and $1$. Thus, the only restricted idempotent additive
group codes of the form $FGe$, where $e\in KG$ is an idempotent, are
$FG0=\{0\}$ and $FG1=FG$. The nonzero code $FG$ has parameters
$(4,2^4,1)$, since it contains the word $1$, which has Hamming weight $1$.
Consequently, no restricted idempotent additive group code has parameters
$(4,2^4,3)$.

We also claim that $C$ is not a classical idempotent group code of the form
$KGe$. Every code of the form $KGe$ is $K$-linear. Let
$c=1+\omega u\in C$. Since $\omega^2=\omega+1$, one has
$\omega c=\omega+\omega^2u=u+\omega(1+u)$. Suppose that $\omega c\in C$.
Then there exists $x\in R$ such that
$\omega c=x+\omega xu$. Comparing the $R$-components gives $x=u$, and
comparison of the $\omega R$-components gives $u^2=1+u$. However,
$u^2=(1+g+g^2)^2=1+g^2+g^4=g^2$, whereas $1+u=g+g^2$. This is a
contradiction. Hence $\omega c\notin C$, so $C$ is not $K$-linear and cannot
be of the form $KGe$.

Finally, $C$ attains the Singleton bound for additive codes. Indeed,
$\dim_{\mathbb{F}_2}(C)=4$, $[K:F]=2$, and $d(C)=3$, so
$d(C)=4-\left\lceil\frac{4}{2}\right\rceil+1$. Equivalently,
$|C|=2^4=4^{4-3+1}$. Hence $C$ is an additive MDS code. This example shows
that, in the non-semisimple case, the projector construction yields an
additive MDS group code with parameters unattainable by both the restricted
idempotent construction $FGe$ and the classical idempotent construction
$KGe$.
\end{example}

In light of Proposition~\ref{proj-summand}, it is natural to ask the following question.

\begin{question}\label{question-1}
Let $C\leq KG$ be an additive left group code, and let
$\star\in\{\mathrm{TE},\mathrm{TH}\}$, where the trace-Hermitian case is
considered when $m$ is even. How can the $\star$-LCD property of $C$ be
characterized in terms of an $FG$-linear projector
$P\in\operatorname{End}_{FG}(KG)$ satisfying
$\operatorname{Im}(P)=C$ and $\ker(P)=C^{\perp_\star}$? Moreover, if
$C=\operatorname{Im}(P)$ for an $FG$-linear projector $P$, under what
conditions on $P$ is $C$ $\star$-self-dual?
\end{question}

In the particular situation of
Proposition~\ref{idempotent-projector}(2), where $C=KGe$, the preceding
question belongs to the classical $K$-linear setting. Since $C$ is a left
ideal of $KG$, it is a $K$-linear code. Hence its Euclidean and
trace-Euclidean duals coincide, and, when $m$ is even, its Hermitian and
trace-Hermitian duals also coincide. Moreover, the full image of the
projector $\rho_e$ is $KGe$, so the duality properties of this code can be
studied through the usual ideal-theoretic methods.

The additive situation is more subtle. Although an idempotent $e\in KG$
defines the projector $\rho_e$, the restricted idempotent additive group code
is $FGe=\rho_e(FG)$, rather than the full image
$\operatorname{Im}(\rho_e)=KGe$. In general, $FGe$ need not be $K$-linear,
and its trace duality properties are not determined solely by the
corresponding duality properties of the classical code $KGe$. Thus, the
relevant problem is to determine the LCD and self-dual behavior of the
restricted image itself.

\begin{question}\label{question-2}
Let $e\in KG$ be an idempotent. Under what conditions is the restricted
idempotent additive left group code $FGe=\rho_e(FG)$ $\star$-LCD or
$\star$-self-dual, where $\star\in\{\mathrm{TE},\mathrm{TH}\}$ and the
trace-Hermitian case is considered when $m$ is even?
\end{question}

\subsection{LCD projector additive group codes}

The aim of this section is to answer the two questions stated above in the $\star$-LCD case. We first characterize $\star$-LCD additive left group codes in terms of self-adjoint $FG$-linear projectors. We subsequently characterize when  the idempotent additive left group code $\rho_e(FG)=FGe$ is $\star$-LCD.

We begin with a general criterion for the image of an $FG$-linear endomorphism to be $\star$-LCD.

\begin{lemma}\label{lem:image-projector-dual-kernel}
Let $T\in \operatorname{End}_{FG}(KG)$, and let $T^\ast$ denote its adjoint with respect to $\langle\cdot,\cdot\rangle_\star$. Then
$\operatorname{Im}(T)^{\perp_\star}=\ker(T^\ast).$
In particular, $\operatorname{Im}(T)$ is a $\star$-LCD additive left group code if and only if
$\operatorname{Im}(T)\cap \ker(T^\ast)=\{0\}.$
\end{lemma}

\begin{proof}
Let $y\in KG$. Then
$y\in \operatorname{Im}(T)^{\perp_\star}$
if and only if
$\langle T(x),y\rangle_\star=0$
for all $x\in KG.$
By the definition of adjoint, this is equivalent to
$\langle x,T^\ast(y)\rangle_\star=0$
for all $x\in KG.$
Since the form is non-degenerate, this holds if and only if $T^\ast(y)=0$. Hence
$\operatorname{Im}(T)^{\perp_\star}=\ker(T^\ast).$
The final assertion follows immediately.
\end{proof}

The following theorem gives a complete answer to Question~\ref{question-1} in the LCD case.

\begin{theorem}\label{thm:LCD-projector-characterization}
Let $C\leq KG$ be an additive left group code. Then the following hold:
\begin{enumerate}
    \item $C$ is TE-LCD if and only if there exists an $FG$-linear projector $P\in \operatorname{End}_{FG}(KG)$ such that $\operatorname{Im}(P)=C$ and $P$ is self-adjoint with respect to $\langle\cdot,\cdot\rangle_{\mathrm{TE}}$.

    \item If $m$ is even, then $C$ is TH-LCD if and only if there exists an $FG$-linear projector $P\in \operatorname{End}_{FG}(KG)$ such that $\operatorname{Im}(P)=C$ and $P$ is self-adjoint with respect to $\langle\cdot,\cdot\rangle_{\mathrm{TH}}$.
\end{enumerate}
\end{theorem}

\begin{proof}
We prove the trace Euclidean case; the trace Hermitian case is identical.

Assume first that $C$ is TE-LCD. Then
$
KG=C\oplus C^{\perp_{\mathrm{TE}}}
$
as a direct sum of $FG$-modules. Let $P:KG\longrightarrow KG$ be the projection onto $C$ parallel to $C^{\perp_{\mathrm{TE}}}$. Then $P\in \operatorname{End}_{FG}(KG)$, $P^2=P$, and $\operatorname{Im}(P)=C$.

To prove that $P$ is self-adjoint, write
$
x=x_1+x_2
$
and
$
y=y_1+y_2,
$
with $x_1,y_1\in C$ and $x_2,y_2\in C^{\perp_{\mathrm{TE}}}$. Then $P(x)=x_1$ and $P(y)=y_1$. Since $C\perp C^{\perp_{\mathrm{TE}}}$, we obtain
$$
\langle P(x),y\rangle_{\mathrm{TE}}
=
\langle x_1,y_1+y_2\rangle_{\mathrm{TE}}
=
\langle x_1,y_1\rangle_{\mathrm{TE}}
=
\langle x_1+x_2,y_1\rangle_{\mathrm{TE}}
=
\langle x,P(y)\rangle_{\mathrm{TE}}.
$$
Thus $P$ is self-adjoint with respect to $\langle\cdot,\cdot\rangle_{\mathrm{TE}}$.

Conversely, assume that there exists an $FG$-linear projector $P\in \operatorname{End}_{FG}(KG)$ such that $\operatorname{Im}(P)=C$ and $P$ is self-adjoint with respect to $\langle\cdot,\cdot\rangle_{\mathrm{TE}}$. Since $P^2=P$, we have
$
KG=\operatorname{Im}(P)\oplus\ker(P).
$
By Lemma~\ref{lem:image-projector-dual-kernel},
$
C^{\perp_{\mathrm{TE}}}=(\operatorname{Im}P)^{\perp_{\mathrm{TE}}}=\ker(P^\ast)=\ker(P),
$
because $P^\ast=P$. Hence
$
KG=C\oplus C^{\perp_{\mathrm{TE}}},
$
and therefore $C$ is TE-LCD.

The trace Hermitian case is proved in exactly the same way.
\end{proof}
The following example illustrates
Theorem~\ref{thm:LCD-projector-characterization}. It exhibits a
trace-Euclidean LCD additive left group code that is the image of both a
self-adjoint and a non-self-adjoint $FG$-linear projector, showing that the
characterization in the theorem is an existence statement.

\begin{example}\label{projector-existencia}
Let $F=\mathbb{F}_2$, let
$K=\mathbb{F}_4=\mathbb{F}_2(\omega)$, where
$\omega^2+\omega+1=0$, and let
$G=C_2=\langle g\mid g^2=1\rangle$. Since
$\operatorname{char}(F)=2$ divides $|G|=2$, the algebra $FG$ is not
semisimple.

Consider the additive left group code $C=FG\omega\subseteq KG$. Since
$\omega$ is central in $KG$, one has $FG\omega=\omega FG$. Moreover,
$K=\mathbb{F}_2\omega\oplus\mathbb{F}_2\omega^2$. Since
$\langle\omega,\omega\rangle_{TE}=
\operatorname{Tr}_{K/F}(\omega^2)=1$ and
$\langle\omega,\omega^2\rangle_{TE}=
\operatorname{Tr}_{K/F}(\omega^3)=\operatorname{Tr}_{K/F}(1)=0$, it follows
that $C^{\perp_{TE}}=FG\omega^2$. Consequently,
$KG=C\oplus C^{\perp_{TE}}$, and hence $C$ is trace-Euclidean LCD. Every element of $KG$ can be written uniquely as $x+\omega y$, where
$x,y\in FG$. Define maps $P_0,P_1:KG\longrightarrow KG$ by
$P_0(x+\omega y)=\omega(x+y)$ and $P_1(x+\omega y)=\omega y$. Since
$\omega$ is central in $KG$, both maps are $FG$-linear. Moreover,
$P_0^2=P_0$, $P_1^2=P_1$, and
$\operatorname{Im}(P_0)=\operatorname{Im}(P_1)=FG\omega=C$. Thus, both
$P_0$ and $P_1$ are $FG$-linear projectors with image $C$. The map $P_0$ is the orthogonal projection onto $C$ along
$C^{\perp_{TE}}=FG\omega^2$. Therefore, $P_0$ is self-adjoint with respect
to the trace-Euclidean form. On the other hand, $P_1$ is not self-adjoint.
Indeed, $P_1(1)=0$ and $P_1(\omega)=\omega$. Hence
$\langle P_1(1),\omega\rangle_{TE}=0$, whereas
$\langle1,P_1(\omega)\rangle_{TE}=
\langle1,\omega\rangle_{TE}=
\operatorname{Tr}_{K/F}(\omega)=1$. Therefore, $P_1^*\neq P_1$. This example shows that the trace-LCD property of $C$ guarantees the
existence of a self-adjoint projector with image $C$, but does not imply that
every projector with image $C$ is self-adjoint.
\end{example}

As an application of Theorem~\ref{thm:LCD-projector-characterization}, we recover the following characterization of left ideals generated by self-adjoint idempotents. Although the full result is already known from \cite{CruzWillems2017,BBW24}, we include the proof as an application of the previous theorem.

\begin{corollary}\label{LCD-idempotent}
Let $C\leq KG$ be a left group code.
\begin{enumerate}
    \item The following are equivalent:
    \begin{enumerate}
        \item $C$ is a Euclidean LCD code.
        \item $C=KGe$ for some idempotent $e\in KG$ such that $e=e^\ast$.
    \end{enumerate}

    \item If $m$ is even, then the following are equivalent:
    \begin{enumerate}
        \item $C$ is a Hermitian LCD code.
        \item $C=KGe$ for some idempotent $e\in KG$ such that $e=\overline{e}^\ast$.
    \end{enumerate}
\end{enumerate}
\end{corollary}

\begin{proof}
We prove the Euclidean case; the Hermitian case is analogous.

Assume first that $C$ is a Euclidean LCD code. By Proposition~\ref{proposition:Ksubspace-orthogonals},
$
C^{\perp_{\mathrm{E}}}=C^{\perp_{\mathrm{TE}}},
$
so $C$ is also TE-LCD. Hence
$
KG=C\oplus C^{\perp_{\mathrm{E}}}
$
as a direct sum of left $KG$-submodules, that is, as a direct sum of left ideals. Let
$
P:KG\longrightarrow KG
$
be the projection onto $C$ parallel to $C^{\perp_{\mathrm{E}}}$. Then $P$ is a left $KG$-module endomorphism with $\operatorname{Im}(P)=C$. By the proof of Theorem~\ref{thm:LCD-projector-characterization}, $P$ is self-adjoint with respect to $\langle\cdot,\cdot\rangle_{\mathrm{TE}}$. By Remark~\ref{adjunto-E-TE}, since $P$ is $K$-linear, its adjoint with respect to $\langle\cdot,\cdot\rangle_{\mathrm{TE}}$ coincides with its adjoint with respect to $\langle\cdot,\cdot\rangle_{\mathrm{E}}$.

Set
$
e=P(1).
$
Since $P$ is $KG$-linear, we have
$
P(x)=xP(1)=xe
$
for all $x\in KG$. As $P^2=P$, it follows that
$
e^2=e,
$
so $e$ is an idempotent. Moreover,
$
C=\operatorname{Im}(P)=KGe.
$

We now show that $e=e^\ast$. Let $x,y\in KG$. Since $P$ is self-adjoint,
$$
\langle xe,y\rangle_{\mathrm{TE}}
=
\langle P(x),y\rangle_{\mathrm{TE}}
=
\langle x,P(y)\rangle_{\mathrm{TE}}
=
\langle x,ye\rangle_{\mathrm{TE}}.
$$
On the other hand, by Lemma~\ref{identidad},
$
\langle xe,y\rangle_{\mathrm{TE}}=\langle x,e^\ast y\rangle_{\mathrm{TE}}.
$
Hence
$
\langle x,e^\ast y\rangle_{\mathrm{TE}}=\langle x,ey\rangle_{\mathrm{TE}}
$
for all $x,y\in KG$. Since $\langle\cdot,\cdot\rangle_{\mathrm{TE}}$ is non-degenerate, it follows that
$
e^\ast y=ey
$
for all $y\in KG$, and therefore $e^\ast=e$.

Conversely, assume that
$
C=KGe
$
for some idempotent $e\in KG$ such that $e=e^\ast$. Consider the map
$
\rho_e:KG\longrightarrow KG$, \;$x\mapsto xe.
$
Then $\rho_e$ is an $FG$-linear projector with $\operatorname{Im}(\rho_e)=C$. Moreover, for all $x,y\in KG$,
$$
\langle \rho_e(x),y\rangle_{\mathrm{TE}}
=
\langle xe,y\rangle_{\mathrm{TE}}
=
\langle x,e^\ast y\rangle_{\mathrm{TE}}
=
\langle x,ey\rangle_{\mathrm{TE}}
=
\langle x,\rho_e(y)\rangle_{\mathrm{TE}},
$$
so $\rho_e$ is self-adjoint with respect to $\langle\cdot,\cdot\rangle_{\mathrm{TE}}$. Hence, by Theorem~\ref{thm:LCD-projector-characterization}, $C$ is TE-LCD. By Proposition~\ref{proposition:Ksubspace-orthogonals}, it follows that $C$ is a Euclidean LCD code.

If $m$ is even, the Hermitian case is proved in the same way, using Proposition~\ref{proposition:Ksubspace-orthogonals}, the TH-version of Theorem~\ref{thm:LCD-projector-characterization}, Remark~\ref{adjunto-E-TE}, and the identity
$
\langle xe,y\rangle_{\mathrm{TH}}=\langle x,\overline{e}^\ast y\rangle_{\mathrm{TH}}.
$
\end{proof}

It is natural to ask whether every projector appearing in Theorem~\ref{thm:LCD-projector-characterization} arises from an idempotent, as in Corollary~\ref{LCD-idempotent}. The following lemma shows that this need not be the case.

First, recall that the corresponding form on $K$ is defined by
$
\langle x,y\rangle_{\mathrm{TE}}=\operatorname{Tr}_{K/F}(xy)
$
in the trace-Euclidean case, and by
$
\langle x,y\rangle_{\mathrm{TH}}=\operatorname{Tr}_{K/F}(x\overline{y})
$
in the trace-Hermitian case.

\begin{lemma}
Let $G$ be a finite group, and suppose that there exists a nonzero proper $F$-subspace $U$ of $K$ such that $K=U\oplus U^{\perp_\star}$, where $\star\in \{\mathrm{TE},\mathrm{TH}\}$. Let $\pi_U:K\longrightarrow U$ be the $F$-linear map determined by $\pi_U(u+w)=u$ for all $u\in U$ and $w\in U^{\perp_\star}$, and define
$$
P\Bigl(\sum_{g\in G}\lambda_g g\Bigr)=\sum_{g\in G}\pi_U(\lambda_g)\,g.
$$
Then $P$ is a self-adjoint projector in $\operatorname{End}_{FG}(KG)$, $C=\operatorname{Im}(P)=UG$ is a $\star$-LCD additive left group code, and $P$ is not induced by right multiplication by any element of $KG$.
\end{lemma}

\begin{proof}
Since $K=U\oplus U^{\perp_\star}$, the map $\pi_U:K\longrightarrow U$ is a well-defined $F$-linear projection with kernel $U^{\perp_\star}$. Hence the map
$$
P\Bigl(\sum_{g\in G}\lambda_g g\Bigr)=\sum_{g\in G}\pi_U(\lambda_g)\,g
$$
is $F$-linear. Moreover, for every $h\in G$ and $x=\sum_{g\in G}\lambda_g g\in KG$, we have
$$
P(hx)=P\Bigl(\sum_{g\in G}\lambda_g(hg)\Bigr)=\sum_{g\in G}\pi_U(\lambda_g)(hg)=h\sum_{g\in G}\pi_U(\lambda_g)g=hP(x).
$$
Thus $P$ is $FG$-linear, so $P\in \operatorname{End}_{FG}(KG)$.

Since $\pi_U^2=\pi_U$, it follows immediately that $P^2=P$, so $P$ is a projector.

We now determine the image and kernel of $P$. For any $x=\sum_{g\in G}\lambda_g g\in KG$, each coefficient of $P(x)$ lies in $U$, so $C=\operatorname{Im}(P)\subseteq UG$. Conversely, if $y=\sum_{g\in G}u_g g\in UG$ with $u_g\in U$ for all $g\in G$, then $P(y)=\sum_{g\in G}\pi_U(u_g)g=\sum_{g\in G}u_g g=y$, since $\pi_U$ restricts to the identity on $U$. Hence $C=\operatorname{Im}(P)=UG$.

Likewise,
$$
P\Bigl(\sum_{g\in G}\lambda_g g\Bigr)=0
\iff
\pi_U(\lambda_g)=0 \text{ for all } g\in G
\iff
\lambda_g\in U^{\perp_\star} \text{ for all } g\in G,
$$
and therefore $\ker(P)=U^{\perp_\star}G$. Hence $KG=UG\oplus U^{\perp_\star}G$.

Now let $x=\sum_{g\in G}u_g g\in UG$ and $y=\sum_{g\in G}w_g g\in U^{\perp_\star}G$, with $u_g\in U$ and $w_g\in U^{\perp_\star}$ for all $g\in G$. Since the form on $KG$ is defined coefficientwise from $\langle\cdot,\cdot\rangle_\star$ on $K$, we obtain
$$
\langle x,y\rangle_\star=\sum_{g\in G}\langle u_g,w_g\rangle_\star=0,
$$
because $U^{\perp_\star}$ is orthogonal to $U$. Thus $UG\perp_\star U^{\perp_\star}G$, so the above decomposition of $KG$ is orthogonal.

Since $P$ is the projection onto $UG$ along $U^{\perp_\star}G$, it follows from this orthogonal decomposition that $P$ is self-adjoint with respect to $\langle\cdot,\cdot\rangle_\star$. Therefore, by Theorem~\ref{thm:LCD-projector-characterization}, the image $C=\operatorname{Im}(P)=UG$ is a $\star$-LCD additive left group code.

Finally, every endomorphism of $KG$ induced by right multiplication by an element of $KG$ is $K$-linear. On the other hand, $P$ is not $K$-linear, because $\pi_U$ is not $K$-linear: indeed, if $\pi_U$ were $K$-linear, then, since $K$ is one-dimensional over itself, its image would be either $\{0\}$ or $K$, contradicting the fact that $U$ is a nonzero proper $F$-subspace of $K$. Hence $P$ cannot be induced by right multiplication by any element of $KG$.
\end{proof}

The following example illustrates the lemma and shows that self-adjoint projectors need not arise from right multiplication by elements of $KG$.

\begin{example}
Let $F=\mathbb{F}_3$, $K=\mathbb{F}_9$, and $G=C_2\times C_2$. Consider on $K$ the trace-Hermitian form
$
\langle\alpha,\beta\rangle_{\mathrm{TH}}=\operatorname{Tr}_{K/F}(\alpha\beta^3).
$
Let $U=\langle 1\rangle_F$. Then $U$ is a nonzero proper $F$-subspace of $K$, and
$$
\langle 1,1\rangle_{\mathrm{TH}}
=
\operatorname{Tr}_{K/F}(1\cdot 1^3)
=
\operatorname{Tr}_{K/F}(1)
=
1+1^3
=
2\neq 0.
$$
Since $U=\langle 1\rangle_F$ is one-dimensional and $\langle 1,1\rangle_{\mathrm{TH}}\neq 0$, the restriction of $\langle\cdot,\cdot\rangle_{\mathrm{TH}}$ to $U$ is nondegenerate. Equivalently, $U\cap U^{\perp_{\mathrm{TH}}}=\{0\}$. Moreover, both $U$ and $U^{\perp_{\mathrm{TH}}}$ have dimension $1$ over $F$, while $\dim_F(K)=2$. Therefore, $K=U\oplus U^{\perp_{\mathrm{TH}}}$.

Thus the lemma applies and yields a self-adjoint projector
$$
P\Bigl(\sum_{g\in G}\lambda_g g\Bigr)=\sum_{g\in G}\pi_U(\lambda_g)\,g
$$
in $\operatorname{End}_{FG}(KG)$ such that $C=\operatorname{Im}(P)=UG$ is a TH-LCD additive left group code. Moreover, $P$ is not induced by right multiplication by any element of $KG$.

The same argument also applies to the trace-Euclidean form, since $\langle 1,1\rangle_{\mathrm{TE}}=\operatorname{Tr}_{K/F}(1)=2\neq 0$.

Additionally, observe that in this case $U = F$, and therefore $C = \operatorname{Im}(P) = UG = FG$. This implies that $C = \operatorname{Im}(P)$ is not a restricted idempotent additive left group code either; that is, $C \neq FGe$ for every idempotent $e \in KG$ with $e \neq 1$.
\end{example}

We now turn to Question~\ref{question-2}. The following result characterizes when an idempotent additive left group code is $\star$-LCD, where $\star\in\{\mathrm{TE},\mathrm{TH}\}$.

\begin{theorem}\label{thm:rhoe-lcd}
Let $e\in KG$ be an idempotent, and let $\rho_e:KG\longrightarrow KG$ be the $FG$-linear projector defined by $\rho_e(x)=xe$. Let $\star\in\{\mathrm{TE},\mathrm{TH}\}$, where the case
$\star=\mathrm{TH}$ is considered only when $m$ is even. Then the following conditions are equivalent:
\begin{enumerate}
    \item The additive left group code $FGe=\rho_e(FG)\subseteq KG$ is $\star$-LCD.

    \item For every $x\in FG$, if $\langle xe,ye\rangle_\star=0$ for all $y\in FG$, then $xe=0$.

    \item The form $\widetilde{B}_e:FG\times FG\longrightarrow F$ defined by
    $
    \widetilde{B}_e(x,y)=\langle xe,ye\rangle_\star
    $
    induces a non-degenerate form on $FG/\ker(\rho_e|_{FG})$, where
    $$
    \ker(\rho_e|_{FG})=\{x\in FG:xe=0\}=FG\cap KG(1-e).
    $$

    \item If $x_1,\dots,x_n$ is any $F$-basis of $FG$ and
    $
    M_e=\bigl(\langle x_i e,x_j e\rangle_\star\bigr)_{1\le i,j\le n},
    $
    then
    $
    \operatorname{rank}(M_e)=\dim_F(FGe).
    $
\end{enumerate}
\end{theorem}

\begin{proof}
We first prove that $(1)$ and $(2)$ are equivalent. Since $\rho_e(FG)=FGe$, the code $FGe$ is $\star$-LCD if and only if $FGe\cap (FGe)^{\perp_\star}=\{0\}$, where the dual is taken in the ambient space $KG$. Every element of $FGe$ is of the form $xe$ for some $x\in FG$. Thus, for $x\in FG$, the condition $xe\in (FGe)^{\perp_\star}$ is equivalent to $\langle xe,ye\rangle_\star=0$ for all $y\in FG$, since the elements of $FGe$ are precisely the vectors $ye$ with $y\in FG$. Therefore, $FGe\cap (FGe)^{\perp_\star}=\{0\}$ if and only if, for every $x\in FG$, the condition $\langle xe,ye\rangle_\star=0$ for all $y\in FG$ implies $xe=0$. This proves $(1)\iff(2)$.

Next we show that $(2)$ and $(3)$ are equivalent. Define $\widetilde{B}_e:FG\times FG\longrightarrow F$ by
$
\widetilde{B}_e(x,y)=\langle xe,ye\rangle_\star.
$
If $x\in\ker(\rho_e|_{FG})$, then $xe=0$, and therefore
$
\widetilde{B}_e(x,y)=\langle xe,ye\rangle_\star=0
$
for all $y\in FG$. Hence
$
\ker(\rho_e|_{FG})\subseteq \operatorname{Rad}(\widetilde{B}_e),
$
so $\widetilde{B}_e$ induces a well-defined form on $FG/\ker(\rho_e|_{FG})$. The induced form is given by
$$
\overline{B}_e(x+\ker(\rho_e|_{FG}),\,y+\ker(\rho_e|_{FG}))=\langle xe,ye\rangle_\star.
$$
To see that $\overline{B}_e$ is well defined, suppose that
$$
x+\ker(\rho_e|_{FG})=x'+\ker(\rho_e|_{FG})
\quad\text{and}\quad
y+\ker(\rho_e|_{FG})=y'+\ker(\rho_e|_{FG}).
$$
Then $x'=x+k_1$ and $y'=y+k_2$ for some $k_1,k_2\in\ker(\rho_e|_{FG})$. Since $k_1e=0$ and $k_2e=0$, we have $x'e=(x+k_1)e=xe$ and $y'e=(y+k_2)e=ye$. Therefore,
$
\langle x'e,y'e\rangle_\star=\langle xe,ye\rangle_\star.
$

Moreover, the induced form $\overline{B}_e$ is non-degenerate if and only if
$
\operatorname{Rad}(\widetilde{B}_e)=\ker(\rho_e|_{FG}).
$
Now, by definition,
$
x\in \operatorname{Rad}(\widetilde{B}_e)
\iff
\langle xe,ye\rangle_\star=0 \text{ for all } y\in FG.
$
Hence condition $(2)$ is equivalent to the inclusion
$
\operatorname{Rad}(\widetilde{B}_e)\subseteq \ker(\rho_e|_{FG}).
$
Since we already know that
$
\ker(\rho_e|_{FG})\subseteq \operatorname{Rad}(\widetilde{B}_e),
$
it follows that
$
\operatorname{Rad}(\widetilde{B}_e)=\ker(\rho_e|_{FG}),
$
which is equivalent to the non-degeneracy of $\overline{B}_e$. Therefore $(2)\iff(3)$.

Finally, we prove that $(3)$ and $(4)$ are equivalent. Let $x_1,\dots,x_n$ be an $F$-basis of $FG$, and let
$$
M_e=\bigl(\langle x_i e,x_j e\rangle_\star\bigr)_{1\le i,j\le n}.
$$
This is the Gram matrix of the form $\widetilde{B}_e$ with respect to the chosen basis. Therefore,
$$
\dim_F(FG)-\operatorname{rank}(M_e)=\dim_F(\operatorname{Rad}(\widetilde{B}_e)).
$$
Condition $(3)$ is equivalent to $\operatorname{Rad}(\widetilde{B}_e)=\ker(\rho_e|_{FG})$, and hence to
$$
\dim_F(FG)-\operatorname{rank}(M_e)=\dim_F(\ker(\rho_e|_{FG})).
$$
Equivalently,
$
\operatorname{rank}(M_e)=\dim_F(FG)-\dim_F(\ker(\rho_e|_{FG})).
$
By the rank-nullity theorem applied to $\rho_e|_{FG}:FG\longrightarrow FGe$, the right-hand side is equal to $\dim_F(FGe)$. Thus $(3)\iff(4)$, and the proof is complete.
\end{proof}

\begin{example}
Let $F=\mathbb{F}_3$, $K=\mathbb{F}_9$, and $G=C_2=\langle g\rangle$. Consider the element $e=2(1+g)\in FG\subseteq KG$. Since $(1+g)^2=1+2g+g^2=2(1+g)$, we obtain $e^2=4(1+g)^2=2(1+g)=e$, so $e$ is an idempotent.

Let $x=a+bg\in FG$, with $a,b\in F$. Then
$$
\rho_e(x)=xe=(a+bg)\,2(1+g)=2(a+b)(1+g),
$$
and therefore $\rho_e(FG)=F(1+g)$. In particular, $\dim_F(\rho_e(FG))=1$.

With respect to the $F$-basis $\{1,g\}$ of $FG$, we have $\rho_e(1)=2(1+g)$ and $\rho_e(g)=2(1+g)$. Hence every entry of the Gram matrix is equal to $\langle 2(1+g),2(1+g)\rangle_{\mathrm{TE}}$. Since the coordinates of $2(1+g)$ in the basis $\{1,g\}$ are $(2,2)$, we get
$$
\langle 2(1+g),2(1+g)\rangle_{\mathrm{TE}}
=
\operatorname{Tr}_{K/F}(2\cdot 2+2\cdot 2)
=
\operatorname{Tr}_{K/F}(2).
$$
As $K/F=\mathbb{F}_9/\mathbb{F}_3$ has degree $2$, we have $\operatorname{Tr}_{K/F}(2)=2+2^3=2+2=1$. Thus
$
M_e=
\begin{pmatrix}
1&1\\
1&1
\end{pmatrix},
$
so $\operatorname{rank}(M_e)=1=\dim_F(\rho_e(FG))$. By Theorem~\ref{thm:rhoe-lcd}, condition $(4)$ holds, and therefore $\rho_e(FG)$ is $\mathrm{TE}$-LCD.

The same argument applies to the trace-Hermitian form. Indeed, the coordinates of $2(1+g)$ lie in $F=\mathbb{F}_3$, so they are fixed by the involution $x\mapsto \overline{x}=x^3$. Hence the trace-Hermitian Gram matrix coincides with $M_e$, and condition $(4)$ of Theorem~\ref{thm:rhoe-lcd} shows that $\rho_e(FG)$ is also $\mathrm{TH}$-LCD.
\end{example}

\begin{remark}\label{rem:rhoe-lcd-and-projector-lcd}
Theorem~\ref{thm:rhoe-lcd} and
Theorem~\ref{thm:LCD-projector-characterization} describe the LCD property
from complementary perspectives. If $FG$ is semisimple, then $FGe$ is
automatically a projector additive left group code. Hence, $FGe$ is
$\star$-LCD if and only if there exists a self-adjoint $FG$-linear projector
$Q$ such that $\operatorname{Im}(Q)=FGe$. Theorem~\ref{thm:rhoe-lcd}
provides conditions for this property that can be verified directly from
$e$. Notice, however, that $Q$ need not be the map
$\rho_e:KG\longrightarrow KG$ defined by $\rho_e(x)=xe$, since
$\operatorname{Im}(\rho_e)=KGe$, whereas $FGe=\rho_e(FG)$.

Theorem~\ref{thm:rhoe-lcd} is also relevant in the non-semisimple case, where
$FGe$ is not a priori a projector additive left group code. If one of the
equivalent conditions of that theorem holds, then $FGe$ is $\star$-LCD.
Consequently, by Theorem~\ref{thm:LCD-projector-characterization}, there
exists a self-adjoint $FG$-linear projector with image $FGe$. Thus, in the
non-semisimple setting, the $\star$-LCD property is a sufficient condition
for $FGe$ to be a projector additive left group code.
\end{remark}

The preceding observation shows that the question of whether a restricted
idempotent additive group code is a projector additive group code is
completely settled in the semisimple case and admits an affirmative answer
for $\star$-LCD codes in the non-semisimple case. This leads naturally to the
following open question.

\begin{question}\label{que:restricted-idempotent-projector}
Let $e\in KG$ be an idempotent, and assume that $FG$ is not semisimple.
Under what conditions is the restricted idempotent additive left group code
$FGe$ a projector additive left group code? In particular, is every
restricted idempotent additive left group code $FGe$ a projector additive
left group code?
\end{question}

\subsection{Self-dual projector additive left group codes}

In this section, we consider the self-dual counterpart of the questions posed above. We first give a sufficient condition for a projector additive left group code to be self-dual in terms of its defining projector. We then specialize this discussion to the classical idempotent setting and finally characterize when a restricted idempotent additive left group code of the form $\rho_e(FG)=FGe$ is $\star$-self-dual.

\begin{proposition}\label{prop:projector-selfdual}
Let $C\le KG$ be a projector additive left group code and $P\in\operatorname{End}_{FG}(KG)$ an $FG$-linear projector such that
$
C=\operatorname{Im}(P).
$
Then the following hold:
\begin{enumerate}
    \item If $P^\ast=I-P$ with respect to $\langle\cdot,\cdot\rangle_{\mathrm{TE}}$, then $C$ is TE-self-dual.
    \item If $m$ is even and $P^\ast=I-P$ with respect to $\langle\cdot,\cdot\rangle_{\mathrm{TH}}$, then $C$ is TH-self-dual.
\end{enumerate}
\end{proposition}

\begin{proof}
Since $P^2=P$, we have $KG=\operatorname{Im}(P)\oplus\ker(P)$, and hence $\ker(I-P)=\operatorname{Im}(P)=C$.

For $\star\in\{\mathrm{TE},\mathrm{TH}\}$, Lemmas~\ref{lem:adjoint-FG-linear} and \ref{lem:adjoint-basic} give
$
C^{\perp_\star}=(\operatorname{Im}(P))^{\perp_\star}=\ker(P^\ast).
$
If $P^\ast=I-P$ with respect to the corresponding form, then
$
C^{\perp_\star}=\ker(I-P)=C.
$

Thus, in case $(1)$ we obtain $C=C^{\perp_{\mathrm{TE}}}$, so $C$ is TE-self-dual. In case $(2)$, assuming $m$ is even, we obtain $C=C^{\perp_{\mathrm{TH}}}$, so $C$ is TH-self-dual.
\end{proof}

The converse of Proposition~\ref{prop:projector-selfdual} does not hold in this generality: the self-duality of $C$ alone does not imply the existence of a projector $Q$ with $\operatorname{Im}(Q)=C$ and $Q^\ast=I-Q$. In the classical idempotent setting, however, the dual can be described explicitly, which leads to a more precise characterization.

\begin{corollary}\label{self-idempo} Let $e\in KG$ be an idempotent and let $C=KGe$. Then the following hold:
\begin{enumerate}
    \item $C$ is self-dual with respect to $\langle\cdot,\cdot\rangle_{\mathrm{TE}}$ if and only if $e^\ast=1-e$.
    \item If $m$ is even, then $C$ is self-dual with respect to $\langle\cdot,\cdot\rangle_{\mathrm{TH}}$ if and only if $\overline{e}^\ast=1-e$.
    \end{enumerate}
\end{corollary}

\begin{proof}
Assume first that $e^\ast=1-e$. Let $\rho_e:KG\longrightarrow KG$ be the $KG$-linear projector defined by $\rho_e(x)=xe$. Then $\operatorname{Im}(\rho_e)=KGe=C$. By Remark~\ref{propiedades-adjunto}, the adjoint of $\rho_e$ with respect to $\langle\cdot,\cdot\rangle_{\mathrm{TE}}$ is $\rho_{e^\ast}$. By Remark~\ref{adjunto-E-TE}, since $\rho_e$ is $K$-linear, its adjoint with respect to $\langle\cdot,\cdot\rangle_{\mathrm{TE}}$ coincides with its adjoint with respect to $\langle\cdot,\cdot\rangle_{\mathrm{E}}$. Since $e^\ast=1-e$, we obtain
$$
\rho_e^{\ast_{\mathrm{TE}}}=\rho_{e^\ast}=\rho_{1-e}=I-\rho_e.
$$
Therefore, Proposition~\ref{prop:projector-selfdual} implies that $C$ is self-dual with respect to $\langle\cdot,\cdot\rangle_{\mathrm{TE}}$.

If $m$ is even and $\overline e^\ast=1-e$, the same argument applies to the trace-Hermitian form.

Conversely, assume that $C$ is self-dual with respect to $\langle\cdot,\cdot\rangle_{\mathrm{TE}}$. Since $C=\operatorname{Im}(\rho_e)$, Lemma~\ref{lem:adjoint-basic}(3) gives
$
C^{\perp_{\mathrm{TE}}}=\ker(\rho_e^{\ast_{\mathrm{TE}}}).
$
By self-duality,
$
C=C^{\perp_{\mathrm{TE}}},
$
and hence
$
KGe=\ker(\rho_e^{\ast_{\mathrm{TE}}}).
$
By Remark~\ref{propiedades-adjunto}, we have
$
\rho_e^{\ast_{\mathrm{TE}}}=\rho_{e^\ast}.
$
Hence, by Proposition~\ref{idempotent-projector},
$$
\ker(\rho_e^{\ast_{\mathrm{TE}}})=\ker(\rho_{e^\ast})=KG(1-e^\ast).
$$

Thus
$
KGe=KG(1-e^\ast).
$
Since $e\in KGe=KG(1-e^\ast)$, there exists $a\in KG$ such that
$
e=a(1-e^\ast).
$
Multiplying on the right by $e^\ast$, we obtain
$
ee^\ast=a(1-e^\ast)e^\ast=0.
$
Similarly, since $1-e^\ast\in KG(1-e^\ast)=KGe$, there exists $b\in KG$ such that
$
1-e^\ast=be.
$
Multiplying on the right by $e$, we get
$
(1-e^\ast)e=be^2=be=1-e^\ast.
$
Using $ee^\ast=0$, this becomes
$
e=1-e^\ast.
$
Hence
$
e^\ast=1-e.
$

If $m$ is even and $C$ is self-dual with respect to $\langle\cdot,\cdot\rangle_{\mathrm{TH}}$, the same argument applies. In this case, Remark~\ref{propiedades-adjunto} gives
$
\ker(\rho_e^{\ast_{\mathrm{TH}}})=KG(1-\overline e^\ast),
$
and self-duality yields
$
KGe=\ker(\rho_e^{\ast_{\mathrm{TH}}})=KG(1-\overline e^\ast).
$
Arguing exactly as above, one obtains
$
\overline e^\ast=1-e.
$
\end{proof}

\begin{remark} For an idempotent $e\in KG$, the condition $e^\ast=1-e$ is equivalent to the two identities $ee^\ast=0$ and $(1-e^\ast)e=1-e^\ast$. Therefore, the result of  Corollary \ref{self-idempo} is exactly analogous to that proved in \cite{DeLaCruzWillems2021} for twisted group codes and, more generally, in \cite{BBW24} for twisted skew group codes. We include it here for completeness and to point out that, in the case of classical group codes, it has no application, since Euclidean self-dual idempotent group codes do not exist; see \cite{Willems2002}.
\end{remark}

Although Euclidean self-dual idempotent group codes do not exist, trace-Euclidean self-dual idempotent additive left group codes, that is, trace-Euclidean self-dual restricted idempotent additive left group codes, do exist, as illustrated by the following example.

\begin{example}

Let $F=\mathbb{F}_2$, let
$K=\mathbb{F}_4=\mathbb{F}_2(\omega)$, where
$\omega^2+\omega+1=0$, and let
$G=C_6=\langle g\mid g^6=1\rangle$. Let
$e=\omega^2g^2+\omega g^4\in KG$.
One checks that $e^2=e$, so $e$ is an idempotent. Hence
$
C=FGe=\mathbb{F}_2C_6e
$
defines an additive code in $KG$. The elements $e,ge,g^2e,g^3e,g^4e,g^5e$ are $\mathbb{F}_2$-linearly independent, so $\dim_{\mathbb{F}_2}(C)=6$ and $|C|=64$. Since $e\in C$ and $\operatorname{wt}(e)=2$, we have $d(C)\le 2$. One verifies that $C$ contains no word of weight $1$, then $d(C)=2$. Thus $C$ has parameters $(6,2^6,2)$ as an additive code over $\mathbb{F}_4$. Moreover, $C$ is trace-Euclidean self-orthogonal in $KG$. Since
$
\dim_{\mathbb{F}_2}(KG)=2|G|=12
$
and
$
\dim_{\mathbb{F}_2}(C)=6,
$
it follows that
$
C=C^{\perp_{\mathrm{TE}}}.
$
Hence $C$ is trace-Euclidean self-dual in $KG$.
\end{example}

We now characterize self-duality for restricted idempotent additive left group codes.

\begin{theorem}\label{thm:rhoe-selfdual}
Let $e\in KG$ be an idempotent, and let $\rho_e:KG\longrightarrow KG$ be the $FG$-linear projector defined by $\rho_e(x)=xe$. Let $\star\in\{\mathrm{TE},\mathrm{TH}\}$, where the case
$\star=\mathrm{TH}$ is considered only when $m$ is even. Then the following conditions are equivalent:
\begin{enumerate}
    \item The restricted idempotent additive left group code $FGe=\rho_e(FG)\subseteq KG$ is $\star$-self-dual.

    \item For all $x,y\in FG$, one has $\langle xe,ye\rangle_\star=0$, and $\dim_F(FGe)=\frac{1}{2}\dim_F(KG)$.

    \item The form $\widetilde{B}_e:FG\times FG\longrightarrow F$ defined by $\widetilde{B}_e(x,y)=\langle xe,ye\rangle_\star$ induces a well-defined form $\overline{B}_e:(FG/\ker(\rho_e|_{FG}))\times (FG/\ker(\rho_e|_{FG}))\longrightarrow F$ given by $$\overline{B}_e(x+\ker(\rho_e|_{FG}),\,y+\ker(\rho_e|_{FG}))=\langle xe,ye\rangle_\star.$$ Moreover, $\overline{B}_e$ is the zero form, and $\dim_F(FGe)=\frac{1}{2}\dim_F(KG).$

    \item If $x_1,\dots,x_n$ is any $F$-basis of $FG$ and $M_e=\bigl(\langle x_i e,x_j e\rangle_\star\bigr)_{1\le i,j\le n}$, then $M_e=0$, and moreover $\dim_F(FGe)=\frac{1}{2}\dim_F(KG)$.
\end{enumerate}
\end{theorem}

\begin{proof}
We prove the equivalences step by step.

\noindent
$(1)\Rightarrow(2)$. Assume that $FGe$ is $\star$-self-dual. Then $FGe=(FGe)^{\perp_\star}$. Hence every pair of elements of $FGe$ is orthogonal with respect to $\langle\cdot,\cdot\rangle_\star$. In particular, for all $x,y\in FG$, we have $\langle xe,ye\rangle_\star=0$. Moreover, since the form is non-degenerate on $KG$ and $FGe=(FGe)^{\perp_\star}$, we have
$
\dim_F(FGe)+\dim_F((FGe)^{\perp_\star})=\dim_F(KG).
$
Thus
$
2\dim_F(FGe)=\dim_F(KG),
$
and hence
$
\dim_F(FGe)=\frac{1}{2}\dim_F(KG).
$

\noindent
$(2)\Rightarrow(1)$. Assume that $\langle xe,ye\rangle_\star=0$ for all $x,y\in FG$, and that $\dim_F(FGe)=\frac{1}{2}\dim_F(KG)$. Then every element of $FGe$ is orthogonal to every element of $FGe$, so
$
FGe\subseteq (FGe)^{\perp_\star}.
$
Since the form is non-degenerate on $KG$, we have
$
\dim_F(FGe)+\dim_F((FGe)^{\perp_\star})=\dim_F(KG).
$
Using the dimension hypothesis, it follows that
$$
\dim_F((FGe)^{\perp_\star})=\frac{1}{2}\dim_F(KG)=\dim_F(FGe).
$$
Hence the inclusion $FGe\subseteq (FGe)^{\perp_\star}$ is an equality, and thus $FGe$ is $\star$-self-dual.

\noindent
$(2)\Leftrightarrow(3)$. Define $\widetilde{B}_e:FG\times FG\longrightarrow F$ by $\widetilde{B}_e(x,y)=\langle xe,ye\rangle_\star$. Since $$\ker(\rho_e|_{FG})=\{x\in FG:xe=0\},$$ if $x-x'\in\ker(\rho_e|_{FG})$ and $y-y'\in\ker(\rho_e|_{FG})$, then $(x-x')e=0$ and $(y-y')e=0$. Hence
$
\langle xe,ye\rangle_\star=\langle x'e,y'e\rangle_\star.
$
Therefore, $\widetilde{B}_e$ induces a well-defined form $$\overline{B}_e:(FG/\ker(\rho_e|_{FG}))\times (FG/\ker(\rho_e|_{FG}))\longrightarrow F$$ given by $\overline{B}_e(x+\ker(\rho_e|_{FG}),\,y+\ker(\rho_e|_{FG}))=\langle xe,ye\rangle_\star$. This induced form is zero if and only if
$
\langle xe,ye\rangle_\star=0
$
for all $x,y\in FG$. Therefore (2) and (3) are equivalent.

\noindent
$(2)\Leftrightarrow(4)$. Let $x_1,\dots,x_n$ be any $F$-basis of $FG$, and let
$$
M_e=\bigl(\langle x_i e,x_j e\rangle_\star\bigr)_{1\le i,j\le n}.
$$
Then $M_e=0$ if and only if $\langle x_i e,x_j e\rangle_\star=0$ for all $1\le i,j\le n$. Since $x_1,\dots,x_n$ is a basis of $FG$, this is equivalent to requiring that $\langle xe,ye\rangle_\star=0$ for all $x,y\in FG$. The dimension condition is the same in (2) and (4). Hence (2) and (4) are equivalent.
\end{proof}


\begin{remark}\label{rem:rhoe-selfdual-and-projector-selfdual}
Theorem~\ref{thm:rhoe-selfdual} and
Proposition~\ref{prop:projector-selfdual} describe self-duality from
complementary perspectives. If $FG$ is semisimple, then $FGe$ is a projector
additive left group code. Hence, there exists an $FG$-linear projector $Q$
such that $\operatorname{Im}(Q)=FGe$. If $Q$ can be chosen so that
$Q^*=I-Q$, then $FGe$ is $\star$-self-dual.

In contrast with the LCD situation discussed in
Remark~\ref{rem:rhoe-lcd-and-projector-lcd}, where the existence of a
self-adjoint projector characterizes the $\star$-LCD property, the condition
$Q^*=I-Q$ is only sufficient for $\star$-self-duality. Thus, the projector
approach gives a complete characterization in the LCD case, whereas it
provides a construction criterion, but not a characterization, in the
self-dual case.

Finally, the condition $e^*=1-e$ characterizes self-duality of the full
$K$-linear idempotent group code $KGe$, but does not in general characterize
self-duality of the restricted code $FGe$.
\end{remark}

 \subsection{Murray--von Neumann equivalence of projectors}

In this subsection, we relate additive left group codes defined by $FG$-linear projectors to Murray--von Neumann equivalence in the endomorphism ring $\operatorname{End}_{FG}(KG)$.

\begin{definition}
Let $A$ be a ring, and let $e,f\in A$ be idempotents. We say that $e$ and $f$ are \emph{Murray--von Neumann equivalent}, and write $e\sim_{\mathrm{MvN}}f$, if there exist $a,b\in A$ such that $e=ba$ and $f=ab$.
\end{definition}

In the present setting, this notion applies to projectors in $\operatorname{End}_{FG}(KG)$ and provides an algebraic characterization of when their images are isomorphic as left $FG$-modules (see Proposition~\ref{prop:equivalent-projectors}).

\begin{lemma}\label{lem:adjoint-neu}
Let $\star\in\{\mathrm{TE},\mathrm{TH}\}$, where the case $\star=\mathrm{TH}$ is considered only when $m$ is even. Let $P,Q\in \operatorname{End}_{FG}(KG)$ be $FG$-linear projectors. Then the following are equivalent:
\begin{enumerate}
    \item $P\sim_{\mathrm{MvN}}Q$.
    \item $P^{\ast_\star}\sim_{\mathrm{MvN}}Q^{\ast_\star}$.
\end{enumerate}
\end{lemma}

\begin{proof}
Assume first that $P\sim_{\mathrm{MvN}}Q$. Then there exist $A,B\in \operatorname{End}_{FG}(KG)$ such that $P=BA$ and $Q=AB$. Taking adjoints with respect to $\langle\cdot,\cdot\rangle_\star$, Lemma~\ref{lem:adjoint-basic}(2) gives
$
P^{\ast_\star}=A^{\ast_\star}B^{\ast_\star}
$
and
$
Q^{\ast_\star}=B^{\ast_\star}A^{\ast_\star}.
$
Hence $P^{\ast_\star}\sim_{\mathrm{MvN}}Q^{\ast_\star}$.

Conversely, assume that $P^{\ast_\star}\sim_{\mathrm{MvN}}Q^{\ast_\star}$. Then there exist $A,B\in \operatorname{End}_{FG}(KG)$ such that
$
P^{\ast_\star}=BA
$
and
$
Q^{\ast_\star}=AB.
$
Taking adjoints again and using Lemma~\ref{lem:adjoint-basic}(1) and (2), we obtain
$
P=(P^{\ast_\star})^{\ast_\star}=A^{\ast_\star}B^{\ast_\star}
$
and
$
Q=(Q^{\ast_\star})^{\ast_\star}=B^{\ast_\star}A^{\ast_\star}.
$
Hence $P\sim_{\mathrm{MvN}}Q$.
\end{proof}




\begin{proposition}\label{prop:equivalent-projectors}
Let $P,Q\in \operatorname{End}_{FG}(KG)$ be $FG$-linear projectors. Then the following are equivalent:
\begin{enumerate}
    \item $\operatorname{Im}(P)\cong \operatorname{Im}(Q)$ as left $FG$-modules.
    \item $P\sim_{\mathrm{MvN}}Q$ in $\operatorname{End}_{FG}(KG)$.
\end{enumerate}
\end{proposition}

\begin{proof}
Assume first that $P\sim_{\mathrm{MvN}}Q$ in $\operatorname{End}_{FG}(KG)$. Then there exist $A,B\in \operatorname{End}_{FG}(KG)$ such that $P=BA$ and $Q=AB$. We claim that the restrictions $A|_{\operatorname{Im}(P)}:\operatorname{Im}(P)\longrightarrow\operatorname{Im}(Q)$ and $B|_{\operatorname{Im}(Q)}:\operatorname{Im}(Q)\longrightarrow\operatorname{Im}(P)$ are mutually inverse $FG$-module isomorphisms.

Let $x\in \operatorname{Im}(P)$. Since $Px=x$, we have
$$
Q(Ax)=AB(Ax)=A(BAx)=A(Px)=Ax,
$$
so $Ax\in \operatorname{Im}(Q)$. Similarly, if $y\in \operatorname{Im}(Q)$, then $Qy=y$, and hence
$$
P(By)=BA(By)=B(ABy)=B(Qy)=By,
$$
so $By\in \operatorname{Im}(P)$.

Moreover, if $x\in \operatorname{Im}(P)$, then $B(Ax)=BAx=Px=x$, and if $y\in \operatorname{Im}(Q)$, then $A(By)=ABy=Qy=y$. Therefore the above restrictions are mutually inverse $FG$-linear isomorphisms, and thus $\operatorname{Im}(P)\cong \operatorname{Im}(Q)$ as left $FG$-modules.

Conversely, suppose that $\varphi:\operatorname{Im}(P)\longrightarrow\operatorname{Im}(Q)$ is an isomorphism of left $FG$-modules, and let $\psi=\varphi^{-1}$. Define maps $A,B:KG\longrightarrow KG$ by $A=\varphi\circ P$ and $B=\psi\circ Q$. Then $A,B\in \operatorname{End}_{FG}(KG)$.

Now let $x\in KG$. Since $Px\in \operatorname{Im}(P)$ and $\varphi(Px)\in \operatorname{Im}(Q)$, we have
$$
BAx=\psi(Q(\varphi(Px)))=\psi(\varphi(Px))=Px.
$$
Thus $BA=P$. Similarly, since $Qx\in \operatorname{Im}(Q)$ and $\psi(Qx)\in \operatorname{Im}(P)$, we obtain
$$
ABx=\varphi(P(\psi(Qx)))=\varphi(\psi(Qx))=Qx.
$$
Thus $AB=Q$, and therefore $P\sim_{\mathrm{MvN}}Q$ in $\operatorname{End}_{FG}(KG)$.
\end{proof}

\begin{corollary}\label{cor:iso-neu}
Let $\star\in\{\mathrm{TE},\mathrm{TH}\}$, where the case $\star=\mathrm{TH}$ is considered only when $m$ is even. Let $C,D\leq KG$ be projector additive left group codes, and write $C=\operatorname{Im}(P)$ and $D=\operatorname{Im}(Q)$, where $P,Q\in \operatorname{End}_{FG}(KG)$ are $FG$-linear projectors. Then the following are equivalent:
\begin{enumerate}
    \item $C\cong D$ as left $FG$-modules.
    \item $P\sim_{\mathrm{MvN}}Q$.
    \item $\ker(P)^{\perp_\star}\cong \ker(Q)^{\perp_\star}$ as left $FG$-modules.
\end{enumerate}
\end{corollary}

\begin{proof}
The equivalence $(1)\Leftrightarrow(2)$ is an immediate consequence of Proposition~\ref{prop:equivalent-projectors}.

By Lemma~\ref{lem:adjoint-neu}, Proposition~\ref{prop:equivalent-projectors}, and Lemma~\ref{lem:adjoint-basic}(4), we have
$$
P\sim_{\mathrm{MvN}}Q
\iff
P^{\ast_\star}\sim_{\mathrm{MvN}}Q^{\ast_\star}
\iff
\operatorname{Im}(P^{\ast_\star})\cong \operatorname{Im}(Q^{\ast_\star})
\iff
\ker(P)^{\perp_\star}\cong \ker(Q)^{\perp_\star}.
$$
Hence $(2)\Leftrightarrow(3)$.
\end{proof}
\begin{remark}\label{rem:mvN-coding-significance}
Corollary~\ref{cor:iso-neu} gives a module-theoretic classification principle
for projector additive left group codes: two codes defined as images of
$FG$-linear projectors are isomorphic as left $FG$-modules if and only if
their defining projectors are Murray--von Neumann equivalent. This may be useful for organizing families of projector additive group codes
and for reducing redundancy in classification or computational searches.
However, Murray--von Neumann equivalence is a module-theoretic rather than a
metric notion. In general, it does not preserve Hamming weight distribution
or minimum distance, and therefore complements the usual notions of code
equivalence.
\end{remark}

\begin{corollary}\label{cor:idempotent-neu}
Let $e,f\in KG$ be idempotents, and let $\rho_e,\rho_f:KG\longrightarrow KG$ be defined by $\rho_e(x)=xe$ and $\rho_f(x)=xf$ for all $x\in KG$. Then the following are equivalent:
\begin{enumerate}
    \item $KGe\cong KGf$ as left $KG$-modules.
    \item $\rho_e\sim_{\mathrm{MvN}}\rho_f$ in $\operatorname{End}_{KG}(KG)$.
    \item $e\sim_{\mathrm{MvN}}f$ in $KG$.
    \item $KGe^\ast\cong KGf^\ast$ as left $KG$-modules.
\end{enumerate}
\end{corollary}

\begin{proof}
Since $\operatorname{Im}(\rho_e)=KGe$ and $\operatorname{Im}(\rho_f)=KGf$, the equivalence $(1)\Leftrightarrow(2)$ follows from Proposition~\ref{prop:equivalent-projectors}.

We next prove $(3)\Rightarrow(1)$. Assume that $e\sim_{\mathrm{MvN}}f$ in $KG$. Then there exist $u,v\in KG$ such that $e=uv$ and $f=vu$. Define maps $\varphi:KGe\to KGf$ and $\psi:KGf\to KGe$ by $\varphi(x)=xu$ for all $x\in KGe$ and $\psi(y)=yv$ for all $y\in KGf$. These maps are well defined. Indeed, if $x=ae\in KGe$, then
$
\varphi(x)=aeu=a(eu).
$
Since $e=uv$ and $f=vu$, we have
$
eu=(uv)u=u(vu)=uf,
$
and therefore
$$
\varphi(x)=a(uf)=(au)f\in KGf.
$$
Similarly, if $y=bf\in KGf$, then
$
\psi(y)=bfv=b(fv).
$
Since
$
fv=(vu)v=v(uv)=ve,
$
it follows that
$$
\psi(y)=b(ve)=(bv)e\in KGe.
$$ Moreover, suppose that $ae=a'e$ for some $a,a'\in KG$. Then  $aeu=a'eu$ i.e.
$\varphi(ae)=\varphi(a'e)$.  Similarly, if $bf=b'f$ for some $b,b'\in KG$, then $\psi(bf)=\psi(b'f)$.
Furthermore, both maps are left $KG$-linear. If $x=ae\in KGe$, then
$$
\psi(\varphi(x))=\psi(xu)=(xu)v=x(uv)=xe=(ae)e=ae=x.
$$
Likewise, if $y=bf\in KGf$, then
$$
\varphi(\psi(y))=\varphi(yv)=(yv)u=y(vu)=yf=(bf)f=bf=y.
$$
Thus $\varphi$ and $\psi$ are mutually inverse left $KG$-module isomorphisms, and therefore $KGe\cong KGf$ as left $KG$-modules.

We now prove $(1)\Rightarrow(3)$. Suppose that $\varphi:KGe\to KGf$ is an isomorphism of left $KG$-modules, and let $\psi:KGf\to KGe$ be its inverse. Set $u=\varphi(e)$ and $v=\psi(f)$. Then $u\in KGf$ and $v\in KGe$. Since $\varphi$ and $\psi$ are left $KG$-linear, for every $a,b\in KG$ we have
$
\varphi(ae)=au
$
and
$
\psi(bf)=bv.
$
Now $\varphi(e)=u$ implies
$
u=\varphi(e^2)=e\varphi(e)=eu,
$
and similarly
$
v=\psi(f^2)=f\psi(f)=fv.
$
Since $u\in KGf$, there exists $a\in KG$ such that $u=af$, and thus
$$
e=(\psi\circ\varphi)(e)=\psi(u)=\psi(af)=av.
$$
Using $u=af$, we obtain
$
uv=(af)v=a(fv)=av=e.
$
Similarly, since $v\in KGe$, there exists $b\in KG$ such that $v=be$, and hence
$$
f=(\varphi\circ\psi)(f)=\varphi(v)=\varphi(be)=bu.
$$
Using $v=be$, we obtain
$
vu=(be)u=b(eu)=bu=f.
$
Therefore $e\sim_{\mathrm{MvN}}f$ in $KG$.

Finally, we prove $(2)\Leftrightarrow(4)$ with respect to $\star=\mathrm{TE}$. By Lemma~\ref{lem:adjoint-neu}, we have
$
\rho_e\sim_{\mathrm{MvN}}\rho_f
$
if and only if
$
\rho_e^{\ast_{\mathrm{TE}}}\sim_{\mathrm{MvN}}\rho_f^{\ast_{\mathrm{TE}}}.
$
By Proposition~\ref{prop:equivalent-projectors}, this is equivalent to
$
\operatorname{Im}(\rho_e^{\ast_{\mathrm{TE}}})\cong \operatorname{Im}(\rho_f^{\ast_{\mathrm{TE}}})
$
as left $KG$-modules. Moreover, by Remark~\ref{propiedades-adjunto},
$
\rho_e^{\ast_{\mathrm{TE}}}=\rho_{e^\ast}
$
and
$
\rho_f^{\ast_{\mathrm{TE}}}=\rho_{f^\ast}.
$
Hence
$
\operatorname{Im}(\rho_e^{\ast_{\mathrm{TE}}})=KGe^\ast
$
and
$
\operatorname{Im}(\rho_f^{\ast_{\mathrm{TE}}})=KGf^\ast.
$
Therefore $(2)\Leftrightarrow(4)$.
\end{proof}

\begin{remark}

In the classical idempotent setting, Corollary~\ref{cor:idempotent-neu} recovers the usual Murray--von Neumann equivalence of idempotents in $KG$ and relates it to the left ideals generated by the adjoint idempotents. More precisely, the last equivalence in Corollary~\ref{cor:idempotent-neu} is formulated with respect to the trace-Euclidean adjoint, so that $KGe^\ast\cong KGf^\ast$. If $m$ is even, the same argument applied to the trace-Hermitian adjoint yields the analogous statement with $KG\,\overline e^\ast\cong KG\,\overline f^\ast$.

\end{remark}

\section{Module duals and quotients by orthogonal codes}

In this section, we connect trace duality for additive group codes with the
module-theoretic duality of left $FG$-modules. More precisely, we show that
the quotient of the ambient group algebra by the orthogonal complement,
$KG/C^{\perp_\star}$, is naturally isomorphic to the module dual
$C^*=\operatorname{Hom}_F(C,F)$. This identifies orthogonal quotients with
dual $FG$-modules and provides a representation-theoretic interpretation of
trace duality. We then specialize this result to projector additive group
codes and classical idempotent group codes.

The field $F$ is naturally a left $FG$-module via the augmentation map. More
precisely, if $x=\sum_{g\in G}a_g g\in FG$ and $\lambda\in F$, then
$x\cdot\lambda=\left(\sum_{g\in G}a_g\right)\lambda$. In particular, each
element $g\in G$ acts trivially on $F$, and this action extends $F$-linearly
to all elements of $FG$.

If $M$ and $N$ are left $FG$-modules, then
$\operatorname{Hom}_F(M,N)$ becomes a left $FG$-module via
$$(g\cdot f)(m)=g\bigl(f(g^{-1}m)\bigr)$$
for all $g\in G$, $f\in\operatorname{Hom}_F(M,N)$, and $m\in M$. This action
extends $F$-linearly to all elements of $FG$.

In particular, taking $N=F$, we obtain that, for every left $FG$-module $M$,
the space $M^*=\operatorname{Hom}_F(M,F)$ is naturally a left $FG$-module.
The left $FG$-module $M$ is said to be \emph{self-dual} if
$M\cong M^*$ as left $FG$-modules.

\begin{remark}
Let $C\leq KG$ be an additive left group code. The dual
$C^*=\operatorname{Hom}_F(C,F)$ is the $F$-linear dual of $C$. In particular,
$C^*$ should not be confused with the set $\{c^*:c\in C\}$ arising from the
involution on $KG$.
\end{remark}

\begin{proposition}
Let $C$ be an additive left group code in $KG$, and let $\langle\cdot,\cdot\rangle_\star$ denote either the trace-Euclidean form or, when $m$ is even, the trace-Hermitian form. Then there is a natural isomorphism of left $FG$-modules $KG/C^{\perp_\star}\cong C^*$.
\end{proposition}

\begin{proof}
Define $\Phi:KG\longrightarrow C^*$ by $\Phi(x)(c)=\langle x,c\rangle_\star$ for all $x\in KG$ and $c\in C$.

Since $\langle\cdot,\cdot\rangle_\star$ is $F$-bilinear, $\Phi$ is $F$-linear. We claim that $\Phi$ is $FG$-linear. It is enough to check this on elements of $G$. Let $g\in G$, $x\in KG$, and $c\in C$. Then
$$
\Phi(gx)(c)=\langle gx,c\rangle_\star.
$$
By Lemma~\ref{identidad}, $\langle gx,c\rangle_\star=\langle x,g^{-1}c\rangle_\star$. Therefore,
$$
\Phi(gx)(c)=\langle x,g^{-1}c\rangle_\star=\Phi(x)(g^{-1}c)=(g\cdot\Phi(x))(c),
$$
since $g$ acts trivially on $F$. Hence $\Phi(gx)=g\cdot\Phi(x)$, and thus $\Phi$ is a homomorphism of left $FG$-modules.

Now $x\in\ker(\Phi)$ if and only if $\Phi(x)(c)=0$ for all $c\in C$, if and only if $\langle x,c\rangle_\star=0$ for all $c\in C$. By definition, this means exactly that $x\in C^{\perp_\star}$. Thus $\ker(\Phi)=C^{\perp_\star}$.

It remains to prove that $\Phi$ is surjective. Let $f\in C^*=\operatorname{Hom}_F(C,F)$. Since $C$ is an $F$-subspace of $KG$, the functional $f$ extends to some $\widetilde{f}\in \operatorname{Hom}_F(KG,F)$.

Now consider the $F$-linear map
$
T:KG\longrightarrow \operatorname{Hom}_F(KG,F)$ defined by $T(x)(y)=\langle x,y\rangle_\star.
$
Because $\langle\cdot,\cdot\rangle_\star$ is nondegenerate, $T$ is injective. Since $KG$ and $\operatorname{Hom}_F(KG,F)$ have the same dimension over $F$, it follows that $T$ is an isomorphism. Therefore there exists $x\in KG$ such that
$
\widetilde{f}(y)=\langle x,y\rangle_\star
$
for all $y\in KG$. Restricting to $C$, we obtain
$$
f(c)=\widetilde{f}(c)=\langle x,c\rangle_\star=\Phi(x)(c)
$$
for all $c\in C$. Hence $f=\Phi(x)$, and thus $\Phi$ is surjective.

By the First Isomorphism Theorem, $KG/C^{\perp_\star}\cong C^*$ as left $FG$-modules.
\end{proof}

\begin{remark}\label{rem:quotient-module-dual-significance}
The isomorphism $KG/C^{\perp_\star}\cong C^*$ identifies the quotient by the
orthogonal code with the $F$-linear dual of $C$ as a left $FG$-module. Thus,
it provides a direct connection between trace duality and the
module-theoretic structure of additive group codes.
\end{remark}

We now reformulate these quotients in terms of $FG$-linear projectors.

\begin{corollary}
Let $C$ be an additive left group code in $KG$, and let $P:KG\longrightarrow KG$ be an $FG$-linear projector such that $C=\operatorname{Im}(P)$. Then the following hold:
\begin{enumerate}
    \item If $P^{\ast_{\mathrm{TE}}}$ denotes the adjoint of $P$ with respect to $\langle\cdot,\cdot\rangle_{\mathrm{TE}}$, then $C^*\cong \operatorname{Im}(P^{\ast_{\mathrm{TE}}})$ as left $FG$-modules.

    \item If $m$ is even and $P^{\ast_{\mathrm{TH}}}$ denotes the adjoint of $P$ with respect to $\langle\cdot,\cdot\rangle_{\mathrm{TH}}$, then $C^*\cong \operatorname{Im}(P^{\ast_{\mathrm{TH}}})$ as left $FG$-modules.
\end{enumerate}
\end{corollary}

\begin{proof}
For the trace-Euclidean form, Lemma~\ref{lem:adjoint-basic}(3) gives $C^{\perp_{\mathrm{TE}}}=(\operatorname{Im}P)^{\perp_{\mathrm{TE}}}=\ker(P^{\ast_{\mathrm{TE}}})$. Hence, by the First Isomorphism Theorem,
$$
KG/C^{\perp_{\mathrm{TE}}}
=
KG/\ker(P^{\ast_{\mathrm{TE}}})
\cong
\operatorname{Im}(P^{\ast_{\mathrm{TE}}}).
$$
On the other hand, by the previous proposition, $KG/C^{\perp_{\mathrm{TE}}}\cong C^*$. Therefore $C^*\cong \operatorname{Im}(P^{\ast_{\mathrm{TE}}})$ as left $FG$-modules.

If $m$ is even, Lemma~\ref{lem:adjoint-basic}(3) gives $C^{\perp_{\mathrm{TH}}}=(\operatorname{Im}P)^{\perp_{\mathrm{TH}}}=\ker(P^{\ast_{\mathrm{TH}}})$. Thus
$$
KG/C^{\perp_{\mathrm{TH}}}
=
KG/\ker(P^{\ast_{\mathrm{TH}}})
\cong
\operatorname{Im}(P^{\ast_{\mathrm{TH}}}).
$$
Since $KG/C^{\perp_{\mathrm{TH}}}\cong C^*$, it follows that $C^*\cong \operatorname{Im}(P^{\ast_{\mathrm{TH}}})$ as left $FG$-modules.
\end{proof}
\begin{corollary}\label{cor:lcd-module-selfdual}
Let $C\leq KG$ be an additive left group code. If $C$ is
$\mathrm{TE}$-LCD, then $C\cong C^*$ as left $FG$-modules. If $m$ is even
and $C$ is $\mathrm{TH}$-LCD, then $C\cong C^*$ as left $FG$-modules.
\end{corollary}

\begin{proof}
Assume first that $C$ is $\mathrm{TE}$-LCD. By
Theorem~\ref{thm:LCD-projector-characterization}, there exists a
trace-Euclidean self-adjoint $FG$-linear projector $P$ such that
$C=\operatorname{Im}(P)$. By the preceding corollary,
$C^*\cong\operatorname{Im}(P^{\ast_{\mathrm{TE}}})$. Since $P$ is
self-adjoint, $P^{\ast_{\mathrm{TE}}}=P$, and hence
$C^*\cong\operatorname{Im}(P)=C$.

The trace-Hermitian case is analogous.
\end{proof}

We now specialize the previous results to the idempotent setting. If $C$ is a $K$-linear code in $KG$, then its Euclidean and trace-Euclidean duals coincide, that is, $C^{\perp_{\mathrm E}}=C^{\perp_{\mathrm{TE}}}$. Likewise, when $m$ is even, its Hermitian and trace-Hermitian duals coincide, namely $C^{\perp_{\mathrm H}}=C^{\perp_{\mathrm{TH}}}$. Thus, for classical idempotent group codes, the preceding results may also be interpreted in terms of the usual Euclidean and Hermitian duals. Since $C=KGe$ is $K$-linear, one may also consider its $K$-dual $\operatorname{Hom}_K(C,K)$, which naturally leads to a $KG$-module formulation. However, in what follows we continue to work with the $F$-dual $C^*=\operatorname{Hom}_F(C,F)$, since this is the dual naturally associated with the trace forms considered throughout the paper.

\begin{corollary}\label{pre-coro}
Let $C=KGe$ be the left ideal of $KG$ generated by an idempotent $e\in KG$, and let $\rho:KG\longrightarrow KG$ be the $KG$-linear projector defined by $\rho(x)=xe$. Then the following hold:
\begin{enumerate}
    \item If $\rho^{\ast_{\mathrm{TE}}}$ denotes the adjoint of $\rho$ with respect to $\langle\cdot,\cdot\rangle_{\mathrm{TE}}$, then $C^*=(KG e)^\ast \cong KGe^\ast$ as left $FG$-modules.

    \item If $m$ is even and $\rho^{\ast_{\mathrm{TH}}}$ denotes the adjoint of $\rho$ with respect to $\langle\cdot,\cdot\rangle_{\mathrm{TH}}$, then $C^*=(KG e)^\ast\cong KG\,\overline e^\ast$ as left $FG$-modules.
\end{enumerate}
\end{corollary}

\begin{proof}
For the trace-Euclidean form, let $\rho^{\ast_{\mathrm{TE}}}$ denote the adjoint of $\rho$ with respect to $\langle\cdot,\cdot\rangle_{\mathrm{TE}}$. Fix $y\in KG$. By the previous Corollary and Remark~\ref{propiedades-adjunto}, we have that
$C^* \cong \operatorname{Im}(\rho_e^{\ast_{\mathrm{TE}}}) = \operatorname{Im}(\rho_{e^{\ast}}) = KGe^{\ast}.$
If $m$ is even, the argument is similar.
\end{proof}

\begin{corollary}
Let $e,f\in KG$ be idempotents, and let $\rho_e,\rho_f:KG\longrightarrow KG$ be defined by $\rho_e(x)=xe$ and $\rho_f(x)=xf$ for all $x\in KG$. Then the following are equivalent:
\begin{enumerate}
    \item $KGe\cong KGf$ as left $KG$-modules.
    \item $\rho_e\sim_{\mathrm{MvN}}\rho_f$ in $\operatorname{End}_{KG}(KG)$.
    \item $e\sim_{\mathrm{MvN}}f$ in $KG$.
    \item $KGe^\ast\cong KGf^\ast$ as left $KG$-modules.
    \item $(KGe)^*\cong(KGf)^*$ as left $KG$-modules.
\end{enumerate}
\end{corollary}

\begin{proof}
  This follows directly from Corollaries~\ref{cor:idempotent-neu} and~\ref{pre-coro}.
\end{proof}

\end{document}